\documentclass[10pt]{amsart}

\usepackage{amsfonts,amsmath,latexsym,amssymb,verbatim,amsbsy,amsthm}
\usepackage{dsfont,bm}
\usepackage{multirow}

\usepackage[top=1in, bottom=1in, left=1in, right=1in]{geometry}

\usepackage[dvipsnames]{xcolor} 

\usepackage{tikz}

\usepackage[colorlinks=true, pdfstartview=FitV, linkcolor=RoyalBlue, citecolor=ForestGreen, urlcolor=blue]{hyperref}

\definecolor{labelkey}{rgb}{0,0,1}

\newtheorem{innercustomgeneric}{\customgenericname}
\providecommand{\customgenericname}{}
\newcommand{\newcustomtheorem}[2]{%
  \newenvironment{#1}[1]
  {%
   \renewcommand\customgenericname{#2}%
   \renewcommand\theinnercustomgeneric{##1}%
   \innercustomgeneric
  }
  {\endinnercustomgeneric}
}
\newcustomtheorem{customthm}{Theorem}
\newcustomtheorem{customlemma}{Lemma}
\newcustomtheorem{customformula}{Formula}
\newcustomtheorem{customproposition}{Proposition}

\theoremstyle{plain}
\newtheorem{THEOREM}{Theorem}[section]

\newtheorem{theorem}[THEOREM]{Theorem}

\newtheorem{lemma}[THEOREM]{Lemma}
\newtheorem{proposition}[THEOREM]{Proposition}

\theoremstyle{definition}

\newtheorem{definition}[THEOREM]{Definition}

\theoremstyle{remark}

\newtheorem{remark}[THEOREM]{Remark}

\newcommand{\thm}[1]{Theorem~\ref{#1}}
\newcommand{\lem}[1]{Lemma~\ref{#1}}

\newcommand{\prop}[1]{Proposition~\ref{#1}}

\newcommand{\sect}[1]{Section~\ref{#1}}

\newcommand{\N}{\ensuremath{\mathbb{N}}}   
\newcommand{\Z}{\ensuremath{\mathbb{Z}}}   
\newcommand{\R}{\ensuremath{\mathbb{R}}}   
\newcommand{\T}{\ensuremath{\mathbb{T}}}   

\def \a {\alpha}
\def \b {\beta}
\def \g {\gamma}
\def \d {\delta}
\def \e {\varepsilon}

\def \k {\kappa}
\def \l {\lambda}
\def \n {\nabla}
\def \s {\sigma}

\def \th {\theta}

\def \D {\Delta}

\def \Th {\Theta}
\def \O {\Omega}

\def \bb {{\bf b}}

\def \bg {{\bf g}}

\def \bk {{\bf k}}
\def \bl {{\bf l}}

\def \bu {{\bf u}}
\def \bv {{\bf v}}   
\def \bw {{\bf w}}

\def \cD {\mathcal{D}}
\def \cE {\mathcal{E}}

\def \cH {\mathcal{H}}
\def \cI {\mathcal{I}}

\def \cM {\mathcal{M}}

\def \cO {\mathcal{O}}
\def \cP {\mathcal{P}}

\def \oS { \overline{S} }

\def \oSr { \overline{Sr} }

\def \ocE {\overline{\mathcal{E}}}

\def \loc {\mathrm{loc}}

\def \one {{\mathds{1}}}

\newcommand{\jap}[1]{\left\langle #1 \right\rangle}
\newcommand{\ave}[1]{ \left[ #1 \right]}

\DeclareMathOperator{\supp}{supp} %

\def \lan {\langle}
\def \ran {\rangle}

\def \p {\partial}

\def \ss {\subset}

\def \GL {Gr\"onwall's Lemma}

\def \CK{Csisz\'ar-Kullback inequality}

\renewcommand{\geq}{\geqslant}

\renewcommand{\leq}{\leqslant}

\def \df  {\, \mbox{d}f}

\def \dv  {\, \mbox{d}v}
\def \dx  {\, \mbox{d}x}

\def \dt  {\, \mbox{d}t}

\def \dy  {\, \mbox{d}y}
\def \dz  {\, \mbox{d}z}

\def \dr  {\, \mbox{d}r}

\def \dw  {\, \mbox{d}w}

\def \dmu  {\, \mbox{d}\mu}

\def \dk  {\, \mbox{d}\kappa}

\def \drho  {\, \mbox{d}\rho}

\def \ddt  {\frac{\mbox{d\,\,}}{\mbox{d}t}}

\def \domain {{\O \times \R^n}}

\def \cMcs {\cM_{\mathrm{CS}}}

\def \cMmt {\cM_{\mathrm{MT}}}

\def \cMfmt {\cM_{\mathrm{\phi}}}

\def \cMseg {\cM_{\mathrm{seg}}}

\def \cMb {\cM_{\b}}

\def\R{\mathbb{R}}
\def\N{\mathbb{N}}
\def\T{\mathbb{T}}
\def\Z{\mathbb{Z}}

\def\vavg{[\bv]_{\rho}}

\def \bbu {\bar{\bu}}
\def \bbv {\bar{\bv}}
\def \bbw {\bar{\bw}}

\def \st {\mathrm{s}}

\begin{document}

\title[Fokker-Planck-Navier-Stokes system]{Unconditional alignment of solutions to the Fokker-Planck-Navier-Stokes system with locally averaged Brinkman force}

\author{Roman Shvydkoy}
\email{shvydkoy@uic.edu}

\author{Trevor Teolis}
\email{tteoli2@uic.edu}

\address{Department of Mathematics, Statistics and Computer Science, University of Illinois at Chicago, Chicago, IL 60607, USA}


\date{\today}


\thanks{\textbf{Acknowledgment.}  
The work of RS  was  supported in part by NSF
grant DMS-2405326 and Simons Foundation}

\maketitle

\begin{abstract}
We study a coupled Fokker-Planck--Navier-Stokes (FPNS) system modeling the dynamics of interacting particles suspended in a viscous incompressible fluid, where the coupling occurs through a locally averaged Brinkman drag force. 
Our main result is the unconditional alignment and synchronization of particle and fluid velocities for all weak solutions, in any dimension, on the periodic domain. 
The proof leverages a new entropy method and a hypocoercivity framework, which together yield quantitative decay estimates and prevent density concentration, a key obstacle in previous analyses.
Our approach applies to a broad class of nonlocal alignment protocols, including the Cucker-Smale model. 
We also prove develop well-posedness theory for global weak solutions with hypoelliptic regularization in any dimension, and global strong solutions in two dimensions.
\end{abstract}



\tableofcontents

\section{Introduction}

The model studied in this note takes origin in the theory of combustion and thin sprays, i.e.  particles or droplets immersed in a fluid with fluid density occupying most of the volume, see \cite{Williams-book, Caf-Papa1983}. It belongs to a class of multiphase systems describing dynamics of the incompressible Navier-Stokes equations, as a model of the ambient fluid, coupled with kinetic description of particles interacting with the fluid through the Brinkman drag force:
\begin{align}
    \label{e:FPNSprior}
    \begin{cases}
        \partial_t f + v \cdot \nabla_x f  = \a \n_v \big( F(f) f \big)+ \beta \nabla_v \cdot \big( (v - \bu) f \big)  + \n_v(\s \n_v f) , \\
        \partial_t \bu + (\bu \cdot \nabla) \bu + \nabla p = \nu \Delta \bu + \gamma  \int_{\R^n} (v - \bu ) f \dv, \\
        \nabla \cdot \bu = 0  , 
    \end{cases}
\end{align}
with $F(f)$ incorporating other internal particle forces such as potential interactions, alignment, etc, while the diffusion coefficient $\s\geq 0$ may generally depend on the macroscopic parameters of the system.  Such systems, called Vlasov- or Fokker-Planck-Navier-Stokes have been the subject of an extensive recent studies, see 
\cite{CarrilloChoiVNS2016, HA-degenerate-diffusion, ChoiKwonEnergyFunctional, kwan-VNS-small-data, BDGM2009, DanShou2025, Dan2024, GJV2004, MV2008} and references therein. A fundamental question to be addressed is synchronization: do  fluid and particle velocities align over time regardless of initial configuration? 

One of the first works on this problem go back Hamdache \cite{H1998} for the Vlasov-Stokes system, but more recently addressed in \cite{DanShou2025, Dan2024, CarrilloChoiVNS2016, HA-degenerate-diffusion, kwan-VNS-small-data} for the type of models we are concerned here. 
These previous works studied the alignment in terms of the energy
\[
	E(t) = \frac 1 2 \int_\domain |v - \bar{\bv}|^2 f \dv \dx + \frac 1 2 \int_\O |\bu - \bar{\bu}|^2 \dx + \frac 1 2 | \bar{\bu} - \bar{\bv}|^2,
\]
where $\bar{\bv}$ and $\bar{\bu}$ denote the averaged velocity of the particles and fluid,
\begin{equation*}
	\label{defn:velocity_averages}
	\bar{\bv} := \int_\domain v f \dv \dx, \qquad \bar{\bu} := \frac{1}{|\Omega|} \int_\Omega \bu \dx. 
\end{equation*}
The (non-coercive) energy law, 
\[
	\ddt E \lesssim -\int_\domain |v - \bar{\bv}|^2 f \dv \dx - \int_\O |\bu - \bar{\bu}|^2 \dx - \int_\domain | v - \bu|^2 f \dv \dx, 
\]
gives rise to partial or conditional aligment results.  
Carrillo, Choi, and Karper~\cite{CarrilloChoiVNS2016}  proved conditional alignment under an $L^{3/2}$ integrability assumption on the particle density for the models with no alignment or noise $\a = \sigma = 0$. For the same model Kwan et al.~\cite{kwan-VNS-small-data}  established total alignment, but for small initial data. 
In connection with the result of \cite{CarrilloChoiVNS2016},
the small data implies that the density remains bounded, which is stronger than the $L^{3/2}$ integrability assumption.
Ha et al.~\cite{HA-degenerate-diffusion} considered the model with multiplicative noise parameter $\s \sim |v - \bar{\bv}|^2$ and alignment force $F$ given by the Cucker-Smale communication protocol $F(f)(t,x,v) = \int_\domain \phi(x-y) (w - v) f(t, y, w) \dw \dy$. In the strong alignment regime, $\a \min \phi > \s$, partial synchronization within the particle and fluid bodies separately is proved on 2D-torus, $\T^2$,
\[ 
	\int_\domain |v - \bar{\bv}|^2 f \dv \dx \to 0, \qquad \int_\O |\bu - \bar{\bu}|^2 \dx \to 0, \qquad \int_\domain | v - \bu|^2 f \dv \dx \to 0. 
\]
Small data synchronization for Vlasov-Navier-Stokes system with no alignment $\a=0$ or noise $\s=0$ was established by Danchin \cite{Dan2024}, and for the same system more recently Danchin and Shou  \cite{DanShou2025} prove  alignment with algebraic rate $\cE \lesssim 1/t$  for large data on $\T^2$ (via an original use of the Trudinger inequality). The convergence rate improves to exponential for small data $\|f_0\|_\infty <\e$. Despite being limited to  two dimensions, this is the only unconditional result on the synchronization problem we could trace.

In this note, we  prove alignment and synchronization for any data and without any restrictions on mutual sizes of forces in 2D and 3D using the  methodology developed for collective dynamics models in \cite{Shv-weak,Shv-EA}.
In fact, it is a consequence of a much stronger and more descriptive asymptotic result which states exponential relaxation of $f$ to the Maxwellian  equilibrium while at the same time the fluid $\bu$ relaxes to a constant hydrostatic state, all centered at a common  vector $\bar{\bw}$.   This result requires two distinct modification of the forces:  a modification of the drag force by local averaging to stabilize oscillations of the fluid velocity, and a special form of the noise which depends on the communication strength between particles, but puts no restrictions on its size.  Specifically, we consider the following variant of \eqref{e:FPNSprior}:
\begin{align}
    \label{e:FPNSe}
    \begin{cases}
        \partial_t f + v \cdot \nabla_x f +  \alpha \nabla_v \cdot ( \st_\rho (\vavg - v) f )  = \beta \nabla_v \cdot \big( (v - \bu_\epsilon) f \big) +  \sigma ( \beta + \alpha \st_\rho) \Delta_v f, \\
        \partial_t \bu + (\bu \cdot \nabla) \bu + \nabla p = \nu \Delta \bu + \gamma  \int_{\R^n}\left( (v - \bu_\epsilon ) f \right)_\epsilon\dv, \\
        \nabla \cdot \bu = 0  . 
    \end{cases}
\end{align}
Here, $f(t,x,v)$ is the particle distribution, $\bu$ and $p$ are the fluid velocity and pressure, $\rho$ and $\rho\bv$ are the macroscopic momentum and density of $f$, respectively. The terms $\st_\rho$ and $\vavg$ encode a local strength and averaging, respectively. For example, the alignment force corresponding to the classical Cucker-Smale protocol, introduced in \cite{CS2007a} would be given  by $\st_\rho = \rho \ast \phi$, $ \ave{\bv}_\rho = \frac{(\bv \rho)\ast \phi}{\rho\ast \phi}$, see \ref{CS} below and \sect{s:EA} for many other examples. The parameters $\alpha, \beta, \gamma, \nu, \sigma$ are arbitrary positive constants, and $\epsilon$ denotes a mollification scale, with $\bu_\e$ denoting the standard mollification. We keep all of these parameters fixed.  Now, let us comment on our particular choice of the noise and the drag force.

The choice of the diffusion coefficient $ \sigma (\b + \a \st_\rho)$ is motivated by the fact that it matches the combined coefficient in front of the $v$-transport: $\b \n_v(v f) + \a \n_v(\st_\rho v f) = (\b + \a \st_\rho) \n_v(v f)$. Consequently, the equation takes a classical form of the Fokker-Planck equation with a drift,
\begin{equation}\label{e:FPbs}
\begin{split}
  \partial_t f + v \cdot &\nabla_x f  =  \bb \cdot \nabla_v  f  +\st \nabla_v \cdot ( \s \n_v f + vf ) ,\\
\bb  &=   -\b \bu_\e- \a \st_\rho \vavg,  \quad  \st = \b + \a\st_\rho.
  \end{split}
\end{equation}
Such a system admits a global Maxwellian equilibrium $\mu$ and the mechanism of relaxation $f \to \mu$ is facilitated through the hypoelliptic structure of the Fokker-Planck force and  the $x$-transport. 

The choice of the locally averaged drag force, however small $\e>0$ might be, is necessary to stabilize the hypocoersivity estimates which  require the norm $\| \n_x \bu\|_\infty$ to remain uniformly bounded. Such a bound in 3D stumbles upon the fundamental issue of global regularity for the Navier-Stokes system, and also in 2D due to lack of any uniform a priori bound. The mollified drag removes this issue by bounding $\| \n_x \bu_\e \|_\infty$ with the total energy controlled at all time.

The system defined this way has a free decaying energy
\begin{equation}
	\ocE =  \sigma  \int_{\domain} f \log \frac{f}{\mu_{\bbv}} \dv \dx  + \frac{\b}{2\g} \int_\Omega |\bu - \bbu|^2 \dx + \frac{\b |\O|}{2(\g +  \b |\O|)} |\bbu - \bbv|^2,
\end{equation}
which in and of itself is suggestive of alignment.  However, as in all the previous attempts surveyed above, working with the energy alone is not sufficient to produce unconditional asymptotic decay $\ocE \to 0$ due to the lack of coercivity in the energy law.  The main new idea exercised in this note is to apply the Bakry-Emery approach  in the spirit of \cite{Shv-EA, Shv-weak}. To achieve this goal, the following key ingredients are addressed:

\begin{itemize}
 \item We prove local existence of classical solutions for thick data (\thm{t:lwp}), and global existence of weak solutions for any data in any dimensions (\thm{t:weak}). Furthermore, we demonstrate, through the global hypoellipticity analysis, that the kinetic component $f$ regularizes in weighted Sobolev spaces of arbitrary order instantaneously (\prop{p:gh}). Although the fluid component $\bu$ may still remain rough, this is sufficient to carry out the hypocoersivity estimates. In 2D on the other hand, the pair $(f,\bu)$ regularizes instantaneously by the classical theory.
   \item The free energy decay $\dot{\ocE} \leq 0$, allows to control the $L\log L$-norm of the particle density $\rho$ uniformly in time which prevents the density from concentrating to a Dirac. This control on concentration is sufficient to restore  the synchronization component $|\bbu - \bbv|$ in the energy law (see \sect{ss:centered-entropy}), a key obstacle in previous analyses.
	\item We establish  an \emph{unconditional alignment} result for a broad class of systems with nonlocal alignment and non-degenerate diffusion, in any dimension, see \thm{thm:main} for the full statement. 
\end{itemize}

In \sect{s:EA} we will discuss the properties of the alignment force needed to prove each of the ingredients listed above. While the list of such properties requires some introduction, let us state here the full result as it applies to the classical Cucker-Smale-based model. We refer to Section \ref{ss:notation} for notation.

\begin{theorem} (FPNS alignment with Cucker-Smale protocol) \label{}
	Consider the Cucker-Smale communication protocol, $\st_\rho = \rho \ast \phi$, $\vavg = (\bv \rho) \ast \phi / (\rho \ast \phi)$, and suppose the kernel is Bochner-positive, i.e. $\phi = \psi \ast \psi$ for some smooth $\psi\geq 0$. Then for  any initial condition $f_0 \in L^\infty$, $|v|^q f \in L^1$, $q \geq n+4$, and $\bu_0\in L^2(\O)$, $\n \cdot \bu_0 = 0$, there exists a global weak solution of Leray-Hopf class satisfying energy inequality
\begin{equation} \label{e:CSenlaw}
	\ddt \ocE 
	  \leq  - \a I_{vv}^{\bv, \rho \ast \phi}  - \b I_{vv}^{\bu_\e} 
	 - \frac \a 2 \int_\O |\bv(x) - \bv(y)|^2 \phi(x-y) \rho(y) \rho(x) \dy \dx  -   \frac{\b \nu}{\g} \int_\Omega |\nabla \bu|^2 \dx.
\end{equation}
  Any such solution is renormalized in the DiPerna-Lions sense for the $f$-component, and for any index $m\in \N$ it regularizes into weighted Sobolev space instantaneously:
	\[
	\|f(t) \|_{H^m_q} \leq C(t_0,T,m), \quad t\in [t_0,T].
\]
Furthermore, for any such solution $(f,\bu)$, $f$ relaxes exponentially fast to a Maxwellian distribution centered around a vector $\bbw$ and $\bu$ converges to $\bbw$:
	\[
	\| f- \mu_{\bbw} \|_1 + \| \bv - \bbw\|_{L^2(\rho)}+ \| \bu - \bbw \|_2 \leq c_1e^{-c_2 t}.
	\]
	The constants $c_1,c_2$ depend only on the initial condition and the parameters of the system. The vector $\bbw$ can be determined from the initial condition 
	\[
	\bbw = \frac{\g \bbv_0 + \b |\O| \bbu_0}{\g + \b |\O|}.
	\]
	\end{theorem}

Let us note some of the features of the energy law for this particular case \eqref{e:CSenlaw} for illustration.
The alignment term controls the particle alignment provided $\rho \ast \phi$ is bounded from below. 
For general models, the control will come from a spectral gap condition on the averaging $[\cdot]_\rho$, developed in \cite{Shv-EA}, which is in turn controlled by a bound from below on an averaged density (the so-called ``thickness'').  
Comparing to the $E$-law for \eqref{e:FPNSprior}, the `drag' term $\int_\domain |v - \bu|^2 f \dv \dx$, which had provided the link to control $|\bbu - \bbv|$, is missing. 
However, the link here is hidden in Fisher information $I_{vv}^{\bu_\e}$, which for instance controls the macroscopic quantity $\int_\O |\bv - \bu_\e|^2 \rho \dx$. 
This gives an idea of how the link can be recovered, but closing the estimate on the full energy requires a hypocoercivity argument, which is presented in Section \ref{ss:hypo}.


\subsection{Notation} \label{ss:notation}

We will assume that the flock has unit mass $\int_\domain f \dv \dx = 1$. 
We denote the total momenta of particles and fluid by
\[
\bar{\bv} =  \int_\domain  v f \dv \dx, \quad \bar{\bu} = \frac{1}{|\O|} \int_\Omega  \bu \dx.
\]
We use notation for  various normalized Maxwellians 
\begin{equation*}\label{}
\begin{split}
\mu & = \frac{1}{|\O|(2\pi\s)^{n/2}} e^{-|v|^2 / 2\s} , \quad \mu_{\rho}  = \frac{\rho(t,x)}{(2\pi\s)^{n/2}} e^{-|v|^2 / 2\s} \\
\mu_{\bg} & = \frac{1}{|\O|(2\pi\s)^{n/2}} e^{-|v - \bg|^2 / 2\s}, \quad \mu_{\rho, \bg}  = \frac{\rho}{(2\pi)^{n/2}} e^{-\frac{|v - \bg|^2} 2}.
\end{split}
\end{equation*}
We denote the basic relative entropy by 
\begin{equation}\label{e:H}
\cH = \sigma  \int_{\domain} f \log \frac{f}{\mu} \dv \dx.
\end{equation}
The partial Fisher information with respect to $\bg$ is denoted by
\[
	I_{vv}^{\bg} = \int_\domain \frac{|\sigma \nabla_v f + (v - \bg) f|^2}{f} \dv \dx,
\]
and, weighted by the strength, 
\[
	I_{vv}^{\bg, \st_\rho} = \int_\domain \st_\rho \frac{|\sigma \nabla_v f + (v - \bg) f|^2}{f} \dv \dx. 
\]
We use $\kappa_\rho$ as a shorthand for the measure $\rho \st_\rho \dx$ and therefore the norm $\| \cdot \|_{L^2(\kappa_\rho)}$ and dot product $(\cdot, \cdot)_{L^2(\kappa_\rho)}$ refer to the $L^2$ norm and dot product with respect to the measure $\kappa_\rho$. 


We use the following notation for weighted Lebesgue spaces
\[
L^p_q : = \left\{ f:  \int_\domain \jap{v}^q f^p \dv \dx <\infty  \right\}, \quad \jap{v} = (1+|v|^2)^{\frac12}.
\]

A proper definition of Sobolev spaces consistent with non-homogeneous nature of the Fokker-Planck operator in \eqref{e:FPbs} must incorporate regressive weights:
\begin{equation}\label{e:Sobdef}
H^{m}_q(\domain) =  \left\{ f :  \sum_{ 2|\bk| + | \bl | \leq 2m}   \int_\domain  \jap{v}^{q - 2|\bk| - | \bl |  } | \p^{\bk}_{x} \p_v^{\bl} f |^2 \dv\dx <\infty \right\}.
\end{equation}
If $q,m$ are large enough, the $H^m_q$-metric controls the Fisher information as well as higher order functional needed in the hypocoercivity estimates, see \cite{TV2000}.

\subsection{A review of environmental averaging models}\label{s:EA}
Let $\O=\T^n$.  Denote by $\cP(\O)$ the set of probability measures on $\O$. An {\em  environmental averaging model} is a family of pairs
\begin{equation}\label{ }
\cM = \{ (\st_\rho, \ave{\cdot}_\rho): \rho \in \cP(\O)\},
\end{equation}
where $\st_\rho: \O \to \R_+$, $\st_\rho \in C(\O)$, represents a (specific) scalar communication strength, 
and $\ave{\cdot}_\rho$ is a family of order-preserving bounded linear operators on  $L^\infty(\rho)$  and $L^2(\rho)$ with $\ave{1}_\rho = 1$. The concept was first introduced in \cite{Shv-EA} however here we will adopt a more narrow version, which postulates that the weighted averages are given in the form of an integral operator
\begin{equation}\label{e:warep}
\bw_\rho=  \st_\rho \ave{\bv}_\rho= \int_\O \phi_\rho(x,y) \bv(y) \drho(y), \quad \rho\text{-a.e.}
\end{equation}
represented by a non-negative communication reproducing kernel $\phi_\rho\in L^1(\drho \otimes \drho)$, $\rho\in \cP$, satisfying
\begin{equation}\label{e:ker-s}
\int_\O \phi_\rho (x,y) \drho(y) = \st_\rho(x), \quad \rho\text{-a.e.}
\end{equation}

The classical example of an environmental averaging comes from the  Cucker-Smale alignment model introduced in \cite{CS2007a,CS2007b}
\begin{equation}\label{e:CS}
\dot{x}_i = v_i, \quad \dot{v}_i =\sum_{j=1}^N m_j \phi(x_i-x_j) (v_j - v_i). 
\end{equation}
Here, $\phi : \R^n \to \R_+$ is a radially symmetric decreasing smooth kernel.
The momentum equation can be expressed as a weighted averaging of flock velocities, 
\[
\dot{v}_i = \sum_{j=1}^N m_j \phi(x_i-x_j) \left( \frac{\sum_{j=1}^N m_j \phi(x_i-x_j)v_j}{\sum_{j=1}^N m_j \phi(x_i-x_j)} - v_i \right).
\]
This corresponds to the model $\cM$ with components given by
\begin{equation}\label{CS}
\st_\rho = \rho \ast \phi, \qquad \ave{\bv}_\rho = \frac{(\bv \rho)\ast \phi}{\rho\ast \phi}.\tag{$\cMcs$}
\end{equation}
The main feature of Cucker-Smale system is that unlike previous results it provides a simple criterion for asymptotic alignment and flocking without reliance on connectivity of the system but just in terms of the communication kernel itself.

\begin{theorem}[\cite{CS2007a,CS2007b,HT2008,HL2009}]\label{t:CSintro}
If the kernel is fat tail, $\int_0^\infty \phi(r) \dr = \infty$, then solutions to \eqref{e:CS} align exponentially fast to the conserved mean velocity $\bar{v} = \frac{1}{ \sum_{j=1}^N m_j} \sum_{j=1}^N m_j v_j$, while the flock remains bounded  
\[
\max_{i=1,\ldots,N} |v_i - \bar{v}| \leq C e^{-\d t}, \qquad \max_{i,j=1,\ldots,N} |x_i - x_j| \leq \bar{D},
\]
where $C,\d,\bar{D}$ depend only on the initial condition and parameters of the kernel. 

Similar statements hold for the kinetic and macroscopic descriptions of  \eqref{e:CS}, see \cite{CFRT2010,TT2014}.
\end{theorem}
 
Since then the system has seen numerous applications to swarming, satellite navigation, control, etc, see accounts \cite{ABFHKPPS,Darwin,VZ2012,MT2014,MP2018,Shv-book,Tadmor-notices}. In certain scenarios the Cucker-Smale protocol is inadequate to describe realistic behavior. Such is the case for instance in heterogeneous flock formations. Motsch and Tadmor introduced in \cite{MT2011,MT2014} a variant of \eqref{e:CS} that rebalances the forces in local clusters more naturally:  
\begin{equation}\label{MT}
\st_\rho = 1, \qquad \ave{\bv}_\rho = \frac{(\bv \rho)\ast \phi}{\rho\ast \phi}.\tag{$\cMmt$}
\end{equation}
One can similarly consider a range of interpolated models between \ref{CS} and \ref{MT}, called $\b$-models, to fit a particular modeling scenario
\begin{equation}\label{Mb}
\st_\rho = (\rho\ast \phi)^\b, \qquad \ave{\bv}_\rho = \frac{(\bv \rho)\ast \phi}{\rho\ast \phi},\quad  \b \geq 0.\tag{$\cMb$}
\end{equation}

A symmetric version of the \ref{MT}-model was implemented in \cite{Shv-hypo} in the study of relaxation and hydrodynamic limit: assuming $\phi \in C^\infty$ and  $\int \phi \dx = 1$,
\begin{equation}\label{Mf}
\st_\rho = 1, \qquad  \ave{\bv}_\rho = \left( \frac{(\bv \rho)\ast \phi}{\rho\ast \phi} \right)\ast \phi .\tag{$\cMfmt$}
\end{equation}

Multi-flock, multi-species, and topological models that fall under this framework were discussed in \cite{HeT2017,ST-multi, ST-topo,Shv-EA}. 

Another important example of a model that is not Galilean invariant, and is instilled into a given landscape of environment $\O$  was introduced in \cite{Shv-EA} as a model of consensus in segregated communities.  Let us fix a smooth partition of unity $g_l \in C^\infty(\O)$, $g_l \geq 0$, and $\sum_{l=1}^L g_l = 1$ subordinated to an open cover $\{ \cO_l \}_{l=1}^L$ of $\O$,  so that $\supp g_l \ss \cO_l$. One defines the model by setting  
\begin{equation}\label{Mseg}
\st_\rho = 1, \quad \ave{\bv}_\rho(x) = \sum_{l=1}^L g_l(x) \frac{\int_\O \bv g_l \rho \dy}{\int_\O g_l \rho \dy} . \tag{$\cMseg$}
\end{equation}

In Table~\ref{t:kernels} we present a summary of reproducing kernels that define each of our core models.
\begin{table}
\begin{center}
\caption{Reproducing kernels}\label{t:kernels}
\begin{tabular}{  c | c | c | c | c | c} 
 Model &     \ref{CS}  &  \ref{MT}  &  \ref{Mb} &  \ref{Mf}  &  \ref{Mseg}  \\
  \hline
$\phi_\rho$ &  $\phi(x-y)$  &  $\displaystyle{ \frac{\phi(x-y)}{\rho\ast \phi(x)} }$ &$\displaystyle{ \frac{\phi(x-y)}{(\rho\ast \phi(x))^{1-\b}} }$ & $\displaystyle{\int_\O \frac{\phi(x-z) \phi(y-z)}{\rho\ast \phi(z)} \dz}$ & $\displaystyle{\sum_{l=1}^L \frac{g_l(x) g_l(y)}{\int_\O \rho g_l \dy}}$  
 \end{tabular}
\end{center}
\end{table}

We will now introduce an important and broad class of models by specifying a list of regularity properties of $\cM$ necessary to develop well-posedess theory for the alignment systems they define. 

First, we assume that all our models are local, meaning that agents always interact within some radius $r_0>0$, called radius of communication:
\begin{equation}\label{e:locker}
\phi_\rho(x, y)  \geq c_0, \text{ for all } |x-y|<r_0, \text{ and all } \rho \in \cP(\O).
\end{equation}
The mass of the flock within a ball of radius of communication defines what we call  {\em local thickness}:
\[ 
\th(\rho,x) = \int_\O \rho(x -r_0 y) \chi(y) \dy,
\]
where $\chi \in C^\infty_0$ is a standard radially symmetric mollifier supported on the unit ball. We also define the {\em global thickness} as
\[
\th(\rho,\O)= \inf_{x\in \O}\th(\rho,x).
\]

Basic conditions on the strength are
\begin{align}
& \sup_{\rho \in \cP(\O)}  \| \st_\rho \|_{\infty}  \leq \oS,  \label{e:s-bdd} \tag{S1}\\
& \rho \to \st_\rho  \text{ is continuous from } L^1(\O) \text{ to } L^\infty(\O),\label{e:s-cont} \tag{S2}\\
& \st_\rho(x)   \geq c\, \th(\rho,x) \quad \text{ for all }  x\in \O. \label{e:s-thick} \tag{S3}
\end{align}


Next, we state a list of  regularity assumptions for uniformly thick flocks.

\begin{definition}\label{d:r}
We say that a model $\cM$ is {\em regular} if in addition to the above it meets the following set of conditions:  for all $\rho, \rho',\rho'' \in L^1(\O)$ we have
\begin{align}
\| \p^k_x \st_\rho \|_\infty + \| \p^k_{x} \phi_\rho \|_\infty +\| \p^k_{y} \phi_\rho \|_\infty& \leq C_{k}(\th(\rho,\O)), \quad k= 0,1,\dots,\label{e:r1} \tag{R1} \\
\| \st_{\rho'} -\st_{\rho''}  \|_\infty  + \| \phi_{\rho'}  -\phi_{\rho''} \|_\infty& \leq C(\th(\rho',\O),\th(\rho'',\O)) \| \rho' - \rho'' \|_1.\label{e:r2}\tag{R2}
\end{align}
\end{definition}

To prove existence of global weak solutions, we need two more continuity features of the model $\cM$.  First, we assume that the kernels belong to the classical Schur class uniformly in $\rho$,
\begin{equation}\label{e:ker-Schur}
\oSr = \sup_{\rho\in \cP} \left\| \int_\O \phi_\rho(x,\cdot) \drho(x) \right\|_\infty <\infty. \tag{Sr}
\end{equation}
Note that the symmetric condition in $y$ is satisfied automatically thanks to \eqref{e:s-bdd} and the stochasticity relation \eqref{e:ker-s}.

Lastly, we require a weak continuity-in-$\rho$ condition  stated as follows: for any sequence of densities $\rho_n \to \rho$ in $L^1$, and any subordinate sequence $|\tilde{\rho}_n| \leq C \rho_n$,  $|\tilde{\rho}| \leq C \rho$, with $\tilde{\rho}_n \to \tilde{\rho}$ in $L^1$, one has
\begin{equation}\label{e:ker-cont}
\int_\O \phi_{\rho_n}(x,\cdot) \tilde{\rho}_n(x) \dx \to \int_\O \phi_{\rho}(x,\cdot) \tilde{\rho}(x) \dx, \text{ weakly$^*$ in } L^\infty(\rho). \tag{wC}
\end{equation}

\begin{definition}
A model satisfying \eqref{e:ker-Schur}-\eqref{e:ker-cont} is said to belong to {\em Schur's class}.
\end{definition}

\begin{lemma}[\cite{Shv-EA,Shv-weak}]\label{l:Schur}
All models \ref{CS}, \ref{MT}, \ref{Mb}, \ref{Mf}, \ref{Mseg} are regular of Schur's class. 
\end{lemma}

The integral operators associated with Schur kernels are  bounded on any $\rho$-weighted spaces  $L^p(\rho)$, $1\leq p\leq \infty$. In particular, for $p=2$, it implies control over the energy of the averagings
\begin{equation}\label{e:en-en}
\| \bw_\rho \|_{L^2(\rho)} \leq \sqrt{\oSr \oS} \| \bu\|_{L^2(\rho)}.
\end{equation}
In the future we will encounter the need for other mapping properties of the averages, in particular from $L^2$ to $L^\infty$, which plays a crucial role in the hypocoersivity analysis.  Let us make a general definition.

\begin{definition}\label{}
We say that a model $\cM$ is  type-$(p,q)$, $1\leq p,q \leq \infty$, if the mapping $\bv \to \bw_\rho$ is bounded from $L^p(\rho)$ to $L^q(\rho)$ uniformly in $\rho$:
\begin{equation}\label{e:pq}
\| \bw_\rho(\bv) \|_{L^q(\rho)} \leq C \|\bv \|_{L^p(\rho)}, \quad \forall \bv \in L^p(\rho), \ \rho \in \cP(\O). 
\end{equation}
\end{definition}
Thus, Schur models are of type $(p,p)$ for all $p\geq 1$. In what follows we encounter the need for models which improve integrability, specifically those of type-$(2,\infty)$. A sufficient condition for type-$(2,\infty)$ can be given in terms of the reproducing kernel, see \cite{Shv-EA}:
\[
\sup_{\rho\in \cP} \left\|  \int_\O  \phi^{2}_\rho(\cdot,y) \drho(y) \right\|_\infty<\infty.
\]
It is satisfied for a broad range of models, in particular, for the Cucker-Smale model and other $\b$-models half way towards Motsch-Tadmor. We summarize these in the following lemma.

\begin{lemma}[\cite{Shv-EA}]\label{}
All models \ref{Mb} for $\frac12 \leq \b \leq 1$, including the classical Cucker-Smale model \ref{CS} are of type-$(2,\infty)$. Furthermore, all models \ref{Mb}, $0 \leq \b \leq 1$, and \ref{Mf} are type-$(2,\infty)$ provided $\phi >0$, while \ref{Mseg}  is  type-$(2,\infty)$ provided $\supp g_l = \O$ for all $l=1,\ldots, L$.
\end{lemma}

\section{Well-posedness}
In this section, we will work under the regularity assumptions on the model $\cM$ stated  in \sect{s:EA}. 

\subsection{Energy and entropy}\label{ss:enen}
We start with the basic energetics of the system. 
Recall that the relative entropy is given in \eqref{e:H}. 
By the \CK, 
\begin{equation*}\label{ }
\cH \geq c \| f - \mu \|_1^2, 
\end{equation*}
the relative entropy controls the distance to the uniformly distributed in $x$ hydrostatic equilibrium.  For convenience, we work with the expanded form
\begin{equation*}
	\cH =  \frac 1 2 \int_{\domain} |v|^2 f \dv \dx +  \sigma \int_{\domain} f \log f \dv \dx + c(\s,n).
\end{equation*}

To establish the $\cE$-law, let us first note that 
\begin{align*}
	\ddt \frac{1}{2} \int_\Omega |\bu|^2 dx = - \nu \int_\Omega |\nabla \bu|^2 dx + \g \int_\domain  \bu_\e \cdot (v - \bu_\e) \df .
\end{align*}
Turning to $\cH$, we have
\begin{align*}
	\ddt \frac{1}{2} \int_\domain |v|^2 \df 
		&= - \int_\domain \Big( \a \st_\rho v \cdot ( v - \vavg) f +   \b v \cdot (v - \bu_\e) f  +  \sigma (\b+\a\st_\rho) v \cdot \nabla_v f \Big) \dv \dx  .
\end{align*} 
Let us rewrite the alignment term as
\begin{equation}
	\label{eqn:alignment_terms}
	\begin{split}
	\int_\domain &\st_\rho v \cdot ( v - \vavg ) \df \\
		&=  \int_\domain \st_\rho v \cdot ( v - \bv ) \df  + \int_\domain \st_\rho v \cdot (\bv - \vavg) \df \\
		&=  \int_\domain \st_\rho | v - \bv |^2 \df + \| \bv \|_{L^2(\kappa_\rho)}^2 - (\bv, \vavg)_{\kappa_\rho} .
	\end{split} 
\end{equation}
Continuing,  
\begin{equation}
	\begin{split}
	\label{law:flogf}
	\ddt  \int_\domain f \log f \dv \dx  
		&= - \int_\domain \sigma (\b + \a \st_\rho) \frac{  |\nabla_v f|^2 }{f} \dv \dx \\
		&\quad - \int_\domain \Big( \a \st_\rho (v - \vavg ) \cdot \nabla_v f  +  \b (v - \bu_\e) \cdot \nabla_v f \Big) \dv \dx . 
	\end{split}
\end{equation}
Using that for any function $\bg:\Omega \to \R^n$, $\int_\domain \bg(x) \cdot \nabla_v f \dv = 0$,  we obtain
\[
\int_\domain \Big( \a \st_\rho (v - \vavg ) \cdot \nabla_v f  + \b (v - \bu_\e) \cdot \nabla_v f \Big) \dv \dx = \int_\domain (\b+ \a\st_\rho) v  \cdot \nabla_v f  \dv \dx.
\]

Adding up all the terms together we obtain the following entropy equation
\begin{align*}
	\ddt \cE &= - \int_\domain \sigma^2 (\b + \a \st_\rho) \frac{  |\nabla_v f|^2 }{f} \dv \dx -\int_\domain  2 \sigma \big(\b + \a\st_\rho \big)   v \cdot \nabla_v f \dv \dx \\
		&\quad - \a \int_\domain \st_\rho | v - \bv |^2 \df - \a \| \bv  \|_{L^2(\kappa_\rho)}^2 + \a (\bv , [\bv ]_\rho)_{\kappa_\rho} \\
		&\quad - \b \int_\domain   |v - \bu_\e|^2 \df -  \frac{\b \nu}{\g} \int_\Omega |\nabla \bu|^2 \dx . 
\end{align*}

We will now group  $\a$-terms and $\b$-terms together. The $\a$-terms give
\begin{equation*}\label{}
\begin{split}
&-\int_\domain \st_\rho \s^2 \frac{  |\nabla_v f|^2 }{f} \dv \dx -\int_\domain \st_\rho 2 \sigma  v \cdot \nabla_v f \dv \dx  -  \int_\domain \st_\rho | v - \bv |^2 \df\\
&-\int_\domain \st_\rho \s^2 \frac{  |\nabla_v f|^2 }{f} \dv \dx -\int_\domain \st_\rho 2 \sigma  (v - \bv) \cdot \nabla_v f \dv \dx  -  \int_\domain \st_\rho | v - \bv |^2 \df\\
& = -I_{vv}^{\bv,\st_\rho}.
\end{split}
\end{equation*}
Similarly, the $\b$-terms give
\begin{equation*}\label{}
\begin{split}
&-\int_\domain  \s^2 \frac{  |\nabla_v f|^2 }{f} \dv \dx -\int_\domain  2 \sigma  v \cdot \nabla_v f \dv \dx  -  \int_\domain  | v - \bu_\e |^2 \df\\
&-\int_\domain  \s^2 \frac{  |\nabla_v f|^2 }{f} \dv \dx -\int_\domain 2 \sigma  (v - \bu_\e) \cdot \nabla_v f \dv \dx  -  \int_\domain  | v - \bu_\e |^2 \df\\
& = -I_{vv}^{\bu_\e}.
\end{split}
\end{equation*}
We obtain 
\begin{equation}\label{e:enlaw}
	\ddt \cE 
	  =  - \a I_{vv}^{\bv, \st_\rho}  - \b I_{vv}^{\bu_\e} 
	 -  \a \| \bv  \|_{L^2(\kappa_\rho)}^2 +\a (\bv , [\bv]_\rho)_{\kappa_\rho}   -   \frac{\b \nu}{\g} \int_\Omega |\nabla \bu|^2 \dx.
\end{equation}

This  fundamental energy law holds for classical solutions as our the computation  is clearly valid in the regularity classes $H^m_q\times H^k$. On the other hand, for weak solutions of Leray-Hopf type, it can be validated as energy inequality, by analogy with the Leray-Hopf weak solutions to the Navier-Stokes equation.

Let us note that for general non-contractive models, such as \ref{MT} or \ref{Mb} we can't conclude that the right hand side of law is negative. Instead by the $L^2$-boundedness of the strength-function and the averages, we have
\[
 -  \a \| \bv  \|_{L^2(\kappa_\rho)}^2 +\a (\bv , [\bv]_\rho)_{\kappa_\rho} \lesssim \|\bv\|^2_{L^2(\rho)} \leq \int_{\domain} |v|^2 f \dv \dx \leq C_1 \cH + C_2 \leq  C_1 \cE + C_2.
 \]
 For the penultimate inequality see \cite[(248)]{Shv-EA}. Thus,
 \begin{equation}\label{e:enCC}
 \ddt \cE \leq C_1 \cE + C_2,
\end{equation}
 and hence, the energy remains bounded on any finite interval of time $[0,T]$.

\subsection{A priori estimates} In this section we combine  some of the classical a priori estimates  of the Navier-Stokes component  together with the new a priori estimates proved in \cite{Shv-EA,Shv-weak} necessary to develop the wellposedness theory of the system. We refer to \cite{Robinson} for the classical theory of the Navier-Stokes component.

Consider the linear Fokker-Planck equation with smooth, bounded coefficients:
    \begin{equation}
      \label{e:linearized problem}
        \begin{split}
            \partial_t f + v \cdot \nabla_x f &= \st(x,t) \nabla_v \cdot (\nabla_v f + v f) + \bb(x,t) \cdot \nabla_v f, \qquad \st \geq c_0 > 0,\  \|\st\|_{C^m} + \|\bb\|_{C^m} \leq B, \\
            \partial_t \bu + \bu \cdot \nabla_x \bu &= \nu \Delta \bu + \gamma \int_{\R^n} \big( (v - \bu_\e)  f \big)_\e \dv.
        \end{split}
    \end{equation}
    
        \begin{lemma}
        	    \label{l:apriori}
	    A solution $(f, \bu) \in H^m_q(\domain) \times H^k(\O)$ to \eqref{e:linearized problem} with $k,m>2$, satisfies
	    \begin{equation}
	    	\label{e:apriori-kin}
	    	\frac{d}{dt} \| f \|_{H^{m}_{q}}^2 + \frac{c_0}{2} \|\nabla_v f \|_{H^{m}_{q}}^2 \leq C_{B} \| f \|_{H^{m}_{q}}^2 ,
	    \end{equation}
	    and 
	    \begin{equation}
	    	\label{e:apriori-NS}
	    	\ddt \|\bu\|_{H^k}^2 + 2 \nu \|\nabla_x \bu\|_{\dot{H}^k}^2 \leq C  \|\bu\|_{H^k}^2\min\{\|\bu\|_{H^k}^2, \|\n_x \bu\|_\infty\} + C_\e  \|\bu\|^2_2 + C_\e\|\bu\|_2 \sqrt{\int_\domain |v|^2 f \dv\dx }.
	    \end{equation} 
    \end{lemma}

    \begin{proof}
    Inequality \eqref{e:apriori-kin} was proved in \cite{Shv-EA}.  We note that this particular estimate is the one where the regressing weights defining $H^{m}_{q}$ are necessary to account for non-homogeneity of the $x$-dependence in the kinetic transport term $\st \n_v \cdot (v f)$.

   To obtain \eqref{e:apriori-NS} we differentiate the equation
    \begin{equation*}
        \partial_t \partial_x^k \bu + \bu \cdot \nabla_x \partial_x^k \bu + \partial_x^k \bu \cdot \nabla_x \bu + \n_x \p^k_x p= \nu \Delta \partial_x^k \bu + \gamma \partial_x^k  \int_{\R^n} \big( (v - \bu_\e)  f \big)_\e \dv,
    \end{equation*}
 Testing with $\partial_x^k \bu$, the Navier-Stokes part can estimated classically in two ways -- either using the enstrophy to bound the non-linearity or, like for the Euler system, using the commutator estimates (see \cite{Robinson}):
 \[
 \ddt \|\bu\|_{H^k}^2 + 2 \nu \|\nabla_x \bu\|_{\dot{H}^k}^2 \leq C  \|\bu\|_{H^k}^2\min\{\|\bu\|_{H^k}^2, \|\n_x \bu\|_\infty\} +J.
 \]
 The drag term  given by
 \[
 J= \gamma \int_\O  \bu \cdot \partial_x^{2k} \int_{\R^n} \big( (v - \bu_\e)  f \big)_\e \dv \, \dx 
\]
can be estimated as 
   \begin{align*}
J \leq C_\e \|\bu\|_{2} \Big\| \int (v - \bu_\e) f \dv \Big\|_{L^1(\Omega)},
   \end{align*}
   and futhermore
   \[
      \Big\| \int (v - \bu_\e) f \dv \Big\|_{L^1(\Omega)} \leq \sqrt{\int_\domain |v|^2 \df} + C_\e\|\bu\|_2.
   \] 
\end{proof}

It will often be useful to bound the particle and fluid energies by the entropy $\cE$. 
\begin{lemma} For some absolute constant $C>0$,
	\label{l:energy-bound}
	\begin{equation}
	    	\label{e:energy-bound}
	    	\|\bu\|_{L^2}^2 + \int_\domain |v|^2 \df \leq 5 \cE + C. 
	    \end{equation} 
	\begin{proof}
	The inequality $\|\bu\|_{L^2}^2 \leq \cE$ is trivial. 
	From the classical inequality, 
   \[
        \int_\domain |f \log f| \dv \dx  \leq \frac{1}{4} \int_\domain |v|^2 \df + \int_\domain f \log f \dv \dx + C,
   \]
   we have that,
   \begin{align*}
    \int_\domain |v|^2 \df 
      &\leq \int_\domain |v|^2 \df + \int_\domain |f \log f| \dv \dx \\
      &\leq \frac{5}{4} \int_\domain |v|^2 \df + \int f \log f \dv \dx + C \\
      &= \frac{3}{4} \int_\domain |v|^2 \df + \cE + C.
   \end{align*}
   and hence 
   \[
        \int_\domain |v|^2 \df \leq 4 \cE + C.
   \]
	\end{proof}
\end{lemma}

If $\cM$ is a regular model, then the smoothness of the coefficients $\st, \bb$ is guaranteed to hold for thick flocks. In fact, such regularity can be controlled by the energy $\cE$. Indeed,
\[
			\| \partial_x^l (\st_\rho \vavg ) \|_\infty \leq \|\partial_x^l \phi \|_\infty \int_\O |\bv(x)| \rho(y) \dy \leq C\int_\domain |v|^2 \df + C \leq C\cE + C,
		\] 
and similarly for the $\bu_\e$ and $\st$.  We therefore obtain the following a priori bounds.

\begin{lemma}
	\label{l:a-priori-thick}
	Let $\cM$ be a regular model of Schur's class. If  $(f, \bu) \in C_w((0,T]; H^m_q(\domain) \times C_w([0,T); H^k(\O))$ is a solution to \eqref{e:FPNSe}, then $f$ satisfies
	 \begin{equation}
	    	\label{ineq:a-priori-kinetic}
	    	\frac{d}{dt} \| f \|_{H^{m}_{q}}^2 + \frac{\beta}{2} \|\nabla_v f \|_{H^{m}_{q}}^2 \leq C_{\epsilon, \beta}(\theta(\rho, \O)) (1 + \cE) \| f \|_{H^{m}_{q}}^2 ,
	    \end{equation}
	    and $\bu$ satisfies
	    \begin{equation}
	    	\label{ineq:a-priori-fluid}
	    	\ddt \|\bu\|_{H^k}^2 + 2\nu \| \n_x \bu\|_{\dot{H}^k}^2 \leq C \|\bu\|_{H^k}^2 \min\{\|\bu\|_{H^k}^2, \|\n_x \bu\|_\infty\}+ C_\e (1+ \cE).
	    \end{equation}
\end{lemma}

\subsection{Local well-posedness for classical thick data}
In this section we address the basic question of local well-posedness of classical solutions. Such a result relies on regularity of coefficients of the kinetic models, which depends (through the regularity assumptions \eqref{e:r1} - \eqref{e:r2}) on the thickness of the flock.  As in  \cite{Shv-EA} we focus on the case of thick initial data.


\begin{theorem}\label{t:lwp}
Suppose the model $\mathcal{M}$ is regular of Schur's class and $\Omega = \mathbb{T}^n$. Let $(f_0, \bu_0) \in H^m_q(\domain) \times H^k(\Omega)$, where $m, q, k \geq n+4$, be initial conditions such that $\th(\rho_0, \Omega) > 0$ and $\nabla \cdot \bu_0 = 0$. Then there exists a unique local solution $(f, \bu)$ to the coupled kinetic-fluid system on a time interval $[0,T]$, where $T > 0$ depends  on the initial energy, thickness, and $\|\bu_0\|_{H^k}$ with
\begin{align*}
  f &\in C_w([0,T]; H^m_q(\domain)), \\
  \nabla_v f &\in L^2([0,T]; H^m_q(\domain)), \\
  \bu &\in C_w([0,T]; H^k(\Omega)) \cap L^2([0,T]; H^{k+1}(\Omega)).
\end{align*}
Moreover, if $(f, \bu) \in L^\infty([0,T); H^m_q(\domain)) \times L^\infty([0,T); H^k(\Omega))$ is a solution such that $\inf_{[0,T)} \th(\rho(t), \Omega) > 0$, $\sup_{[0,T)}\| \bu\|_{H^1} < \infty$, then the solution can be extended beyond $T$ in the same class.
\end{theorem}
\begin{proof}

We consider the Leray projections onto the modes residing in the frequency ball $B_N = \{k\in \Z^n: |k|\leq N\}$, $P_N : L^2 \to L^2$, $\widehat{ P_N \bu } = \one_{B_N} \hat{\bu}$. We denote by $\bu^N$ velocities with Fourier support in $B_N$.   We consider the approximated system,
\begin{equation}
\label{eqn:approximated-system}
\begin{cases}
  \partial_t f^{N} + v \cdot \nabla_x f^{N}  =  \bb^{N} \cdot \nabla_v  f^{N}  +\st^N \nabla_v \cdot (  \n_v f^{N} + vf^{N} ) ,\\
   \partial_t \bu^N + P_N(\bu^N \cdot \nabla \bu^N) + \nabla p^N = \nu \Delta \bu^N + \g P_N\left( \int_{\R^n}\left( (v - (\bu^N)_\e ) f^{N}  \right)_\e\dv \right)\\
        \nabla \cdot \bu^N = 0.
\end{cases}
\end{equation}
Note that the Navier-Stokes component is a Galerkin system of ODEs, however, it is coupled with the kinetic PDE so the existence of solutions must be elaborated. 
We can construct solutions to \eqref{eqn:approximated-system} as a limit of the iteration scheme:
\begin{equation}
\label{e:nN}
\begin{cases}
  \partial_t f^N_n + v \cdot \nabla_x f^N_n  =  \bb^N_{n-1} \cdot \nabla_v  f^N_n  +\st^N_{n-1} \nabla_v \cdot (  \n_v f^N_n + vf^N_n ) ,\\
   \partial_t \bu^N_n + P_N(\bu^N_n \cdot \nabla \bu^N_n) + \nabla p^N_n = \nu \Delta \bu^N_n + \g P_N\left( \int_{\R^n}\left( (v - (\bu^N_{n-1})_\e ) f^N_n  \right)_\e\dv \right)\\
        \nabla \cdot \bu^N_n = 0 . 
\end{cases}
\end{equation}
Here, $\bu_0^N = \bu^N_n(t=0) = P_N(\bu_0)$, and $f_n(t=0) = f_0$ is given by the initial condition. The scheme is solved in the following order: first we solve the kinetic equation to obtain $f_n$, which is then used passively to solve the fluid equation as a system of forced ODEs. 

To comment on the well-posedness of the linear kinetic component we recall that for any pair of smooth coefficients $\bb, \st \in C^\infty([0,T)\times \O)$, and $\st \geq c_0$, the Fokker-Planck equation,
\[
 \partial_t f + v \cdot \nabla_x f  =  \bb \cdot \nabla_v  f  +\st \nabla_v \cdot (  \n_v f + vf ),
\]
is classically well-posed in any $H^{M}_Q$, see \cite{Krylov-book}.	In fact, this can by obtained elementary by observing that the viscous approximation
\[
 \partial_t f^R + v \cdot \nabla_x f^R  =  \bb \cdot \nabla_v  f^R  +\st \nabla_v \cdot (  \n_v f^R + \chi(v/R) v f^R ) + \frac1R \D_x f,
 \]
 generates an analytic $C_0$-family of operator as the transport operator is relatively compact compared to the elliptic generator $\st \D_v  + \frac1R \D_x$, see \cite{EngelNagel}; and one has the following bound in $H^{M}_Q$ 
  \[
\ddt \|f^R\|_{H^M_Q}^2\leq C\|f^R\|_{H^M_Q}^2 - c \|\n_v f^R\|_{H^M_Q}^2 ,
\]
with $c,C>0$ independent of $R$ but only depending on the smoothness of parameters, see for example \cite{Shv-EA}. Thus, one obtains a solution in the limit $R \to \infty$ by the standard compactness argument.

We now need to make sure that the sequence of local solutions on each step exists on a common interval $[0,T)$ with uniform bounds in $n$.  To this end, we argue by induction in $n$. First, let us denote $2 \d = \Th(\rho_0,\O)$. Second, as a consequence of the regularity assumption  \eqref{e:r1}, we have a general thickness-dependent bound
\begin{equation}\label{e:22infty}
\| \bw_\rho\|_\infty \leq C(\th(\rho,\O)) \|\bv\|_{L^2(\rho)}  \leq C(\th(\rho,\O)) \sqrt{\cE} + C(\th(\rho,\O)).
\end{equation}
On the initial step, $n=1$, we have constant in time coefficients $\bb_0^N, \st_0^N$  for the kinetic part, and hence according to the above, the solution will exist globally in time $T_1 = \infty$. With a given $f_1^N$, the Galerkin system yields a classical  global solution as the drag force remains bounded on any finite time interval. This completes the initial step. 

To complete the induction, let us make a priori estimates on the energy
\[
\cE^N_n =   \int_{\domain} f^N_n \log \frac{f^N_n}{\mu} \dv \dx + \frac{\b}{2\g} \int_\Omega |\bu_n^N|^2 \dx.
\]
 Let us assume that on some interval of time $[0,T_{n-1})$, the thickness of the flock on the previous step remains substantial $\th(\rho_{n-1}^N,\O) >\d$. Then, the system \eqref{e:nN} is well-posed on $[0,T)$, and according to  \eqref{e:22infty} on that interval we have
\[
\| \bw_{n-1}^N\|_\infty \leq  C(\d) \sqrt{\cE_{n-1}^N} + C(\d).
\]
By a similar computation as in \sect{ss:enen},
\begin{equation}\label{ineq:nth-entropy}
\begin{split}
\ddt \cE^N_n \leq  &- \a I_{vv}^{\bv^N_n, \st_{\rho^N_{n-1}}}  - \b I_{vv}^{(\bu^N_n)_\e} 
	  -   \frac{\b \nu}{\g} \int_\Omega |\nabla \bu^N_n|^2 \dx\\
	  & -  \a \int_\O \st_{\rho^N_{n-1}} |\bv^N_n|^2 \rho_n \dx +\a \int_\O \bw^N_{n-1} \cdot \bv^N_n \rho^N_n \dx \\
	 & + \int_\O (\bu^N_n)_\e \cdot ( (\bu^N_n)_\e - (\bu^N_{n-1})_\e ) \rho^N_n \dx.
\end{split}
\end{equation}
Dropping all the negative terms we estimate the residual energies. First, the fluid energy,
\[
 \int_\O (\bu^N_n)_\e \cdot ( (\bu^N_n)_\e - (\bu^N_{n-1})_\e ) \rho_n \dx \leq \| (\bu^N_n)_\e\|_\infty^2 + \| (\bu^N_n)_\e\|_\infty\| (\bu^N_{n-1})_\e\|_\infty \leq C_\e \cE^N_n + C_\e \sqrt{\cE^N_{n-1} \cE^N_n}.
 \]
Next, the alignment energy,
\[
\int_\O \bw^N_{n-1} \cdot \bv^N_n \rho^N_n \dx \leq \|\bv^N_n\|_{L^2(\rho_n)} \| \bw^N_{n-1}\|_\infty \leq C(\d) (\sqrt{\cE^N_{n-1} \cE^N_n} + \cE^N_n +1).
\] 

 Finally, we arrive at a recursive relation
 \[
 \ddt \cE^N_n \lesssim C_{\e,\d} \cE^N_n + C_{\e,\d} \sqrt{\cE^N_{n-1} \cE^N_n} + C_{\e,\d}.
 \]
 Denoting $\bar{\cE}^N_n = \max\{ \cE^N_0,\ldots,\cE^N_n\}$ we obtain the new equation for the maximal energy by Rademacher's lemma
 \[
 \ddt \bar{\cE}^N_n  \leq C_{\e,\d} \bar{\cE}^N_n + C_{\e,\d}.
 \]
 Given than $\bar{\cE}^N_n(0) = \cE^N_0\leq \cE_0$, we obtain for all $t<T_{n-1}$,
 \begin{equation}\label{e:enn}
\bar{\cE}^N_n (t) \leq   \cE_0 e^{t C_{\e,\d} } + e^{t C_{\e,\d} } -1 .
\end{equation}
So, by further restricting the interval to $t< \frac{\ln 2}{C_{\e,\d}}$ we obtain
\[
\bar{\cE}^N_n (t) \leq 2  \cE_0  + 1.
\]
This translates this into a  lower bound on the thickness which satisfies 
\[
\p_t \th(\rho^N_n,x) = - ( \bv^N_n \rho^N_n)\ast \n \psi_{r_0} (x) \geq - c_1 \cE^N_n - c_2 \geq - c_1 (2  \cE_0  + 1) - c_2 =: -c_3,
\]
 Thus, since $\th(\rho_0,\O) =2\d $, 
\begin{equation}\label{ }
\th(\rho^N_n(t),\O) \geq 2\d - t c_3 > \d,
\end{equation} 
provided $t< \d / c_3$.  So, setting
\[
T_n = \min\{ T_{n-1}, \d / c_3, \frac{\ln 2}{C_{\e,\d}}\}
\]
the same condition on thickness is satisfied on the new time interval. Recalling that $T_0 = \infty$, we conclude that the estimates stabilize on the interval independent of $n,N$ on the time interval 
\[
T = \min\{ \d / c_3, \frac{\ln 2}{C_{\e,\d}}\}
\]
which depends only on the initial thickness and the energy. In particular, $\th(\rho_n^N,\O) >\d$ on $[0,T)$, which translates into 
\[
	\| \partial_x^k \bb^N_n \|_{\infty} \leq C_{k,\d} ( \| \bv^N_n \|_{L^2(\rho_n)} + \| (\bu_n^N)_\e \|_2 ) \leq C_{k,\d,\e} \cE^N_n, \qquad \| \partial_x^k \st^N_n \|_{\infty} \leq C_{k,\d},
\]
Hence, the kinetic equation in \eqref{e:nN} admits a solution on $[0,T)$  with uniformly bounded $H^m_q$-norms. For the same reason, $\bu_n^N$ remain uniformly bounded in $n$.  

Finally, to construct solutions to the approximated system \eqref{eqn:approximated-system}, we use a-priori estimates to get compactness of the sequence--- in particular, using Lemma \ref{l:a-priori-thick} along with the control on the thickness on the energy and the Sobolev embedding $H^{k-1} \hookrightarrow L^\infty$ to get for $t \in [0,T]$, we get: 
 \begin{equation}
 	\begin{split}
		\frac{d}{dt} \| f^N_n \|_{H^{m}_{q}}^2 + \beta \|\nabla_v f^N_n \|_{H^{m}_{q}}^2 &\leq  C_{\epsilon, \beta, T} \| f^N_n \|_{H^{m}_{q}}^2 , \\
		\ddt \|\bu_n^N\|_{H^k}^2 + 2 \nu \|\nabla_x \bu_n^N\|_{\dot{H}^k}^2 &\leq C \|\bu_n^N\|_{H^k}^3 + C_{\e, T} .
	\end{split} 
\end{equation}
By \GL, we conclude that $f^N_n \in L^{\infty}([0,T]; H^m_q(\domain))$. Further, 
\begin{align*}
	 \int_\domain |\partial_t f^N_n|^2 \dv \dx 
	 	&\leq 2 \int_\domain |v|^2 | \nabla_x f^N_n |^2 \dv \dx + 2 \int_\domain |\bb^N_{n-1}|^2 |\nabla_v f^N_n|^2 \dv \dx \\
	 &\quad+ 2 \int_\domain |\st^N_{n-1}|^2 \big( |\Delta_v f^N_n|^2 + |v|^2 (f^N_n)^2 \big) \dv \dx \leq C \| f_n \|_{H^m_q}^2,
\end{align*}
and hence $\partial_t f^N_n \in L^2([0, T]; L^2(\domain))$.  Since $H^m_q \ll H^{m-1}_{q-1} \hookrightarrow L^2$, the Aubin-Lions lemma yields a strongly convergent subsequence $f^N_n \to  f^N \in C([0,T]; H^{m-1}_{q-1}(\domain))$.  For $m$ large enough, the smoothness of $\bu_n^N$ and the classical estimate this implies convergence of each term and hence $f^N$ is a solution to \eqref{eqn:approximated-system}. Since the fluid estimate is a Ricatti equation, we have $\bu_n^N \in L^{\infty}([0,T]; H^k(\O))$.
This implies that $\bu_n^N \rightharpoonup^* \bu^N$ in $L^{\infty}([0,T]; H^k(\O))$. 
Note that the thickness remains uniformly bounded from below on a common interval $\th(\rho^N(t),\O) \geq \d$ for all $N\in \N$ on $[0,T)$. The latter implies that the coefficients of the kinetic part remains uniformly in all $C^l$-classes on $[0,T)$, which in turn implies uniform bounds on $\|f^N\|_{H^m_q}$ by \lem{l:a-priori-thick}. By further restricting the interval $[0,T)$ if necessary and using \eqref{ineq:a-priori-fluid} we obtain also a uniform bound on $\|\bu\|_{H^k}$ on the same time interval.  Note that $T$ depends only on parameters of the system and $\th_0$, $\cE_0$, and $\|\bu_0\|_{H^k}$.

To obtain solutions to the original system, we follow the classical compactness of argument: $f^N \to f \in C([0,T]; H^{m-1}_{q-1}(\domain))$ and $\bu^N \to \bu \in C([0,T]; H^{k-1}(\O))$.  We note that since $\partial_t f \in L^2([0, T]; L^2(\domain))$, we have $f \in C_w([0,T]; L^2(\domain))$, and since $L^2$ is dense in $(H^m_q)^{-1}$, we have weak continuity at the top regularity class $f \in C_w([0,T]; H^m_q(\domain))$. Similarly, $\bu \in C_w([0,T]; H^k(\Omega))$.

The extension statement of the theorem follows  from the estimates above and the 
classical boostrap control of higher order Sobolev norms of the Navier-Stokes solutions via $H^1$, see \cite{Robinson} (note that the drag force plays no role here due to its smoothing properties, \eqref{ineq:a-priori-fluid}).

Uniqueness follows from the kinetic argument carried out in \cite{Shv-EA} for the Fokker-Plank-alignment models combined with the classical estimates on the Navier-Stokes part. 
\end{proof}

\subsection{Existence of weak solutions}

Extension of local solutions constructed in the previous section stumbled upon two major obstacles -- first is of course the classical lack of global existence for solutions of the 3D Navier-Stokes system, and second a possible formation of vacuum or loss of thickness of the particle component. In fact if the particle component is not thick initially then the Sobolev estimates cannot hold unless the model $\cM$ has some thickness-indepenent regularity properties.  Nevertheless, a global existence of weak solutions is achievable for general regular Schur models, as was shown in \cite{KMT2013, Shv-weak}, and we can show the same for the coupled system \eqref{e:FPNSe}.  What is important for subsequent asymptotic analysis is that the constructed solutions are of Leray-Hopf class, so they still satisfy the {\em energy inequality} in differential distributional-in-time form.

\begin{theorem}\label{t:weak}
Let $\cM$ be a regular model of Schur's class. For any initial condition $f_0 \in L^\infty \cap L^1_q$, $q \geq 2$, and $\bu_0\in L^2(\O)$, $\n \cdot \bu = 0$, there exists a global weak solution to \eqref{e:FPNSe} satisfying the following conditions

\underline{Regularity}:  The solution belongs to the class
\begin{equation}\label{e:regweak}
\begin{split}
f  &\in L^\infty([0,T);L^\infty \cap L^1_q) \cap C([0,T); \cD'(\domain)), \quad  \frac{|\n_v f|^2}{f} \in L^1([0,T) \times \domain),\\
\bu & \in L^\infty([0,T); L^2(\O)) \cap L^2([0,T); H^1(\O)) \cap C([0,T); \cD'(\O)),
\end{split}
\end{equation}
for any $T>0$. 

\underline{Energy inequality}: The solution satisfies the following energy equality in distributional-in-time sense on $[0,\infty)$
\begin{equation}\label{e:enineq}
	\ddt \cE 
	  \leq  - \a I_{vv}^{\bv, \st_\rho}  - \b I_{vv}^{\bu_\e} 
	 -  \a \| \bv  \|_{L^2(\kappa_\rho)}^2 +\a (\bv , [\bv]_\rho)_{\kappa_\rho}   -   \frac{\b \nu}{\g} \int_\Omega |\nabla \bu|^2 \dx.
\end{equation}

\underline{Renormalization}: Any such weak solution satisfies the kinetic equation in the renormalized sense of DiPerna-Lions, see \prop{p:renorm} for the full statement.
\end{theorem}

The key first step towards obtaining global weak solutions is to construct an approximation scheme that respects the energy law \eqref{e:enlaw} and provides  compactness of the approximating sequence on any time interval $[0,T)$.
Fix a $\d>0$ and $N\in \N$, and artificially lift the flock density to make it uniformly thick:
\begin{equation*}\label{}
\begin{split}
\rho^{\d} & = \frac{\rho + \d}{1+\d} \in \cP(\O) \\
\bb^{\d, N}  &=   -\b \bu^N_\e- \a \bw^{\d},  \quad  \st^{\d} = \b + \a\st_{\rho^{\d}},\\
 \bw^\d(x)  & = \int_\O \phi_{\rho^{\d}}(x,y) \bv(y) \rho(y) \dy.
\end{split}
\end{equation*}
Note that the disconnect between the $\rho^{\d}$ defining the kernel and the $\rho$ in the momentum $\bv(y) \rho(y)$ is intentional, as the pairing between the velocity and the residual constant $\d / 1+\d$, albeit small, would not be under control.

Let us consider the approximation system
\begin{equation}\label{e:FPNSdN}
\begin{cases}
  \partial_t f^{\d,N} + v \cdot \nabla_x f^{\d,N}  =  \bb^{\d,N} \cdot \nabla_v  f^{\d,N}  +\st^\d \nabla_v \cdot (  \n_v f^{\d,N} + vf^{\d,N} ) ,\\
   \partial_t \bu^N + P_N(\bu^N \cdot \nabla \bu^N) + \nabla p^N = \nu \Delta \bu^N + \g P_N\left( \int_{\R^n}\left( (v - \bu^N_\e ) f^{\d,N}  \right)_\e\dv \right)\\
        \nabla \cdot \bu^N = 0  
\end{cases}
\end{equation}
subject to mollified initial condition $\bu_0^N = P_N \bu_0$, and 
\[
f_0^{\d,N}(x,v) = \chi(x,\d v) f_0 \ast \chi_\d (x,v),
\]
where $\chi$ is a standard compactly supported mollifier. Note that 
\[
\| f_0^{\d,N} \|_{L^\infty \cap L^1_q} \leq \| f_0 \|_{L^\infty \cap L^1_q},
\]
and $f^{\d,N}_0 \in H^{M}_Q(\domain)$ for any $M,Q\in \N$.  Note that \eqref{e:FPNSdN} is a hybrid PDE/ODE system with the state space $(f,\bu)\in H^{M}_Q \times \R^{n(N)}$. 

\begin{lemma}\label{lmma:well-posedness-approximation-system}
System \eqref{e:FPNSdN} is globally well-posed in the class  $(f,\bu) \in H^{M}_Q(\domain) \times \R^{n(N)}$.
\end{lemma}

This follows from the proof of \thm{t:lwp} and the fact that the thickness remains bounded from below independent of the solution (thanks to the  lifting of the flock's density to $\rho_\delta$).

Let us now go back to the approximate system \eqref{e:FPNSdN} and investigate compactness of the family $(f^{\d,N}, \bu^N)$. 
The system \eqref{e:FPNSdN} obeys the same a priori energy bound as in \eqref{e:enlaw} thanks to the easy to check Schur's property of the kernel $\phi_{\rho^\d}$ relative to the unscaled density $\rho$, see \cite{Shv-weak}.  This results in an a priori bound 
\[
\sup_{t\in [0,T)} \cE^{\d,N}(t)  + \int_0^T \| \n_x \bu^N \|_2^2 \dt\leq C_{T}.
\]
Next, by the classical Leray theory, in dimension $n=3$,
\[
\| \bu^N \cdot \nabla \bu^N\|_{H^{-1}} \leq c \|\bu\|_2^{1/2} \| \n \bu \|_2^{3/2},
\]
and for the drag force,
\begin{equation*}\label{}
\begin{split}
\lan P_N \int_{\R^n}\left( (v - \bu^N_\e ) f^{\d,N}   \right)_\e\dv, \phi \ran & = \lan \int_{\R^n}(v - \bu^N_\e ) f^{\d,N}  dv, P_N \phi_\e \ran \leq \left\| \int_{\R^n}(v - \bu^N_\e ) f^{\d,N}  dv\right\|_1 \|(P_N \phi)_\e\|_\infty\\
&  \leq c_\e ( \| \bv^{\d,N}\|_1 + \| \bu^N_\e \rho^{\d,N} \|_1) \|P_N \phi\|_2 \leq c_\e (\sqrt{\cE^{\d,N}} + \|\bu_\e^N\|_\infty) \|\phi\|_2\\
& \leq C_{\e,T} \|\phi\|_2.
\end{split}
\end{equation*}
So,
\[
\left\| \int_{\R^n}\left( (v - \bu^N_\e ) f^{\d,N}   \right)_\e\dv \right\|_{H^{-1}}  \leq \left\| \int_{\R^n}\left( (v - \bu^N_\e ) f^{\d,N}   \right)_\e\dv \right\|_2 \leq C_{\e,T}.
\]
In other words, the drag force is not contributing any harmful term for the $H^{-1}$-bound on the right hand side of the Galerkin system. We thus obtain
\begin{equation}\label{ }
\int_0^T \| \p_t \bu^N\|_{H^{-1}}^{4/3} \dt \leq C_T.
\end{equation}
Since the drag-force estimate is dimension independent, and nonlinearity improves for $n=2$ we have a similar classical estimate in 2D as well.

Moving on to the kinetic part, an a priori estimate for solutions  is given directly by the  maximum principle,
\[
\sup_{t\in [0,T)}  \|f^{\d,N}\|_\infty \leq C_{T}.
\]

With a view towards application of the averaging lemma, let us write the Fokker-Planck equation as follows
\begin{equation}\label{e:FPAg}
\p_t f^{\d,N} + v\cdot \n_x f^{\d,N}  =  \D_v (g^{\d,N}_1) + \n_v \cdot g^{\d,N}_2,
\end{equation}
where 
\[
g^{\d,N}_1 = \st^\d  f^{\d,N}, \quad g^{\d,N}_2 = \st^\d v f^{\d,N} + \bb^{\d,N} f^{\d,N}.
\]
As shown in \cite{Shv-weak}, and the drag force contributing a bounded element only $\bu^N_\e$, we have $g^{\d,N}_1, g^{\d,N}_2 \in L^2_{t,x,v}$ uniformly in $\d,N$. Applying the averaging lemma to the sequence $f^{\d,N}$ and Aubin-Lions Lemma to the sequence $\bu^N$ we extract a subsequence with macroscopic quantities $\rho_\psi = \int_{\R^n} \psi(v) f^{\d,N} \dv$ converging strong in $L^2_{t,x,v}$ while $\bu^N$ converging strong in $L^2_{t,x}$ and weak in $L^\infty_t L^2_t$ and $\n \bu^N$ weakly in $L^2_{t,x}$. Denote the limit by $(f,\bu)$. As shown in  \cite{Shv-weak} the alignment and Fokker-Planck forces converge to the natural limits (this is where the assumptions \eqref{e:ker-cont} and \eqref{e:s-cont} are required), while the Navier-Stokes  nonlinearity converges classically. The only novelty is the drag force, which presents no issues: for any test-function $\psi$ 
\[
\lan  f^{\d,N},   \bu_\e^N \cdot \nabla_v \psi \ran = \lan \rho^{\d,N}_{\n_v \psi} , \bu^N_\e \ran.
\]
Since $\rho^{\d,N}_{\n_v \psi} \to \rho_{\n_v \psi}$ in $L^2_{t,x}$ and $\bu^N_\e \to \bu_\e$ in $L^2_{t,x}$ (or even in any $H^k$!) we clearly obtain
\[
\lan  f^{\d,N},   \bu_\e^N \cdot \nabla_v \psi \ran \to \lan  f,   \bu_\e \cdot \nabla_v \psi \ran.
\]
At the same time for any $x$-dependent test-function $\psi$,
\begin{equation*}\label{}
\begin{split}
\lan P_N \int_{\R^n}\left( (v - \bu^N_\e ) f^{\d,N}   \right)_\e\dv, \phi \ran & = \lan \int_{\R^n}(v - \bu^N_\e ) f^{\d,N}  dv, P_N \phi_\e \ran \\
& = \lan \bv^{\d,N} \rho^{\d,N} , P_N \phi_\e \ran - \lan \bu^N_\e \rho^{\d,N} , P_N \phi_\e \ran.
\end{split}
\end{equation*}
As before, $\| P_N \phi_\e  \|_\infty \leq c_\e \|P_N \phi \|_2 \leq \| \phi\|_2$, while as established in \cite[(69)]{Shv-weak} we have strong convergence of momenta $\|\bv^{\d,N} \rho^{\d,N} - \bv \rho\|_1 \to 0$ and densities $\| \rho^{\d,N} - \rho\|_1 \to 0$. Clearly, the above paring converges to the natural limit.


Next, the energy inequality follows distributionally  in time as in the classical theory for the $\bu$-component. The only novelty again is the drag force which by strong convergence of components implies 
\[
\lan P_N \int_{\R^n}\left( (v - \bu^N_\e ) f^{\d,N}   \right)_\e\dv, \bu^N  \ran = \lan \int_{\R^n}\left( (v - \bu^N_\e ) f^{\d,N}   \right)_\e\dv, \bu_\e^N  \ran \to \lan \int_{\R^n}\left( (v - \bu_\e ) f   \right)_\e\dv, \bu_\e  \ran.
\]
As to the $f$-component is follows from the entropy \emph{equality} satisfied by all weak solutions as follows from the renormalization proved in \cite{Shv-weak}. Let us state it in the context of the drag force.

\begin{proposition}\label{p:renorm}
Suppose $G \in C^\infty_\loc((0,\infty)) \cap C_\loc([0,\infty))$, $G(0) = 0$, and 
\[
\sup_{0<x<X} x |G''(x)| <C(X), \quad \forall X>0.
\]
Then 
\begin{equation*}\label{}
\begin{split}
G(f), G(f)\bb  & \in L^1(\domain),\\
\st_\rho \n_v f G'(f), \ v f G'(f) &\in L^1(\domain \times [0,T])\\
 \st_\rho |\n_v f|^2 G''(f), \st_\rho \n_v f v f G''(f) & \in L^1(\domain \times [0,T]),\end{split}
\end{equation*}
and the renormalized form of equation  \eqref{e:FPbs} 
\begin{equation}\label{ }
\begin{split}
\p_t G(f) + v \cdot \n_x G(f) + \bb \cdot \n_v G(f) & = \st \n_v \cdot (\s \n_v G(f) + v f G'(f)) \\
&- \st (\s \n_v f + v f) \cdot \n_v f G''(f)\\
\bb  =   -\b \bu_\e- \a \st_\rho \vavg,  \quad  \st &= \b + \a\st_\rho
\end{split}
\end{equation}
holds in distributional sense on test-functions from $W^{1,\infty}(\domain \times [0,T))$.
\end{proposition}

We point out again that due to the regularity of the drag force the approximation procedure used in  \cite{Shv-weak} converges weakly to the limiting transport term $\b \bu_\e \cdot \n_v G(f)$.

Preservation of the $L^1_q$-moment follows by the weak lower-semi-continuity and the corresponding propagation of the moment for the approximate sequence. The details are identical to \cite{Shv-weak} in this case because   the drag force contributes only  
\begin{equation*}\label{}
\begin{split}
\int_\domain \bu_\e^N \n_v f^{\d,N} \jap{v}^q \dv \dx & = - \int_\domain \bu_\e^N  f^{\d,N} \n_v \jap{v}^q \dv \dx \leq \sqrt{\cE^{\d,N}} \int_\domain f^{\d,N} \jap{v}^{q-1} \dv \dx \\
& \lesssim \int_\domain f^{\d,N} \jap{v}^q\dv \dx.
\end{split}
\end{equation*}

\subsection{Gain of positivity and global hypoellipticity}

It is crucial for the hypocoercivity analysis to have a solution with sufficient regularity in weighted Sobolev spaces. In fact, such regularity is required only for the kinetic component $f$, and not the fluid velocity $\bu$ which is inaccessible in 3D as is well-known. So, in this section we will illustrate how to gain more regularity of $f$ coming  from the hypoelliptic structure of the equation.

To gain regularity from the hypoelliptic structure we need to acquire regularity of the coefficients in the kinetic equation, which for regular models $\cM$ relies  on the  thickness of the flock.  The thickness can be regained from the spread of positivity effect that's is natural for solutions of the Fokker-Planck type models, see \cite{Kolm1934,Hormander1967,FranPoli2006,GIMV2019,DesVill2000,HST2020,Mouhot2005,IMS2020}. Specifically for the alignment models, the following result was proved in  \cite{Shv-EA} for classical solutions and translated to weak solutions in \cite{Shv-weak}.  
Since the nature of the drift $\bb$ has not been used in the proof, it carries over to our  Fokker-Planck-Navier-Stokes system ad verbatim.

\begin{proposition}\label{p:Gauss} Let $\cM$ be a regular model of Schur's class. 
For a given weak solution $(f,\bu)$ on a  time interval $[0,T)$ as stated in \thm{t:weak} there exists $a,b>0$ which depend only on the parameters of the model $\cM$, time $T$, and 
\begin{equation}\label{e:BH}
\begin{split}
B & = \sup_{t\in [0,T)} \| \bb \|_\infty,\\
H &=  \sup_{t\in [0,T)}  \int_{\T^n \times \R^n} |v|^2 f \dv \dx + \int_{\T^n \times \R^n}  f | \log f| \dv \dx,
\end{split}
\end{equation}
such that
\begin{equation}\label{e:Gauss}
f(t,x,v) \geq b e^{-a |v|^2} , \qquad \forall (t,x,v) \in \T^n \times\R^n \times [T/2,T).
\end{equation}
\end{proposition}

\begin{remark}\label{}
 It should be noted that here the thermalization coefficient is bounded from below $\st \geq \b >0$ by design. So, the construction of the initial plateau and the elliptic region is not necessary in this case. However, the renormalization stipulated in  \prop{p:renorm}  is crucial for the application of the weak maximum principle for supersolutions used in the proof.
\end{remark}

By the classical bound, $H \leq C_1 \cH + C_2$ for some absolute $C_1,C_2>0$. So, $H$ will remain bounded on any finite time interval as a consequence of the energy inequality \eqref{e:enineq}. As to $B$, the fluid component $\bu_\e$ is of course bounded by $c_\e \|\bu\|_2$ which in turn is bounded by $\cE$, so it too remains bounded on any finite time interval. The only component which is not a priori bounded is the environmental averaging $\bw_\rho = \st_\rho \ave{\bv}_\rho$. This is where we invoke the type-$(2,\infty)$ property stated in \eqref{e:pq} as a tool to control the drift.

Since $T>0$ can be chosen arbitrary in \prop{p:Gauss} it applies instantaneously and on an arbitrarily long period of time $[t_0, T)$, with $a,b>0$ being dependent of $t_0, T>0$ of course. This makes the density bounded from below $\rho \geq c(a,b)$, and hence uniformly thick $\th(\rho,\O) \geq c(a,b,r)$.  Thanks to the regularity of the model $\cM$, \eqref{e:r1}, all the coefficients of the kinetic equation become $C^\infty$. At this point \cite[Theorem 5.1]{Shv-weak} applies.

\begin{proposition}\label{p:gh}  Let $\cM$ be a regular type-$(2,\infty)$ model of Schur's class.  Then for any weak solution $(f,\bu)$ as stated in \thm{t:weak}, for any $m\in \N$ and  any $t_0,T>0$ there exists a constant $C(t_0,T,m)>0$ such that 
\begin{equation}\label{e:gh}
\|f(t) \|_{H^m_q} \leq C(t_0,T,m), \quad t\in [t_0,T].
\end{equation}
\end{proposition}

Let us note that if $m \gg q$, then the weights in front of the higher order derivatives of $f$ are inverted as seen from the definition \eqref{e:Sobdef}.  So, in this case \eqref{e:gh} shows that all higher order derivatives  are still bounded on any compact subdomain $\O \times B_R$ but may grow algebraically at $v$-infinity.

It should be noted, again, that the gain of regularity we obtained on the kinetic part has no bearing on regularity of the fluid part as it will depend on the Navier-Stokes regularity itself, and not on the drag force.  As a consequence we obtain a classical global well-posedness in 2D. We state the result without proof as the bootstrap from $\bu\in L^2$ into higher $H^k$ proceeds classically.

\begin{proposition} \label{p:2d_strong}
  Let $n = 2$ and let $\cM$ be a regular type-$(2,\infty)$ model of Schur's class.  Then for any weak solution $(f,\bu)$ as stated in \thm{t:weak}, for any $m\in \N$, $k\in \N$, and  any $t_0,T>0$ there exists a constant $C(t_0,T,m,k)>0$ such that  
  \begin{equation}\label{e:ghu}
\|f(t) \|_{H^m_q} + \|\bu(t)\|_{H^k} \leq C(t_0,T,m,k), \quad t\in [t_0,T].
\end{equation}
\end{proposition}

\section{Alignment and synchronization}

\subsection{Functional properties of the models necessary for syncronization}

\begin{definition}\label{d:cons}
We say that the model $\cM$ is {\em conservative} if 
\begin{equation}\label{e:doublestoch}
\int_\O \phi_\rho (x,y) \drho(x)  = \st_\rho(y).
\end{equation}
In terms of averaging properties, this is equivalent to a statement that
for any  $\rho \in \cP(\O)$, $\bv\in L^2(\k_\rho)$
\begin{equation}\label{e:cons}
\int_{\O}  \bv   \dk_\rho = \int_{\O} \ave{\bv}_\rho \dk_\rho.
\end{equation}
\end{definition}
Among the models in our list, the conservatives ones are \ref{CS}, \ref{Mf}, and \ref{Mseg}.

The dynamics under conservative models for basic alignment systems preserves the total momentum. This can be seen, for example, for solutions of the Fokker-Planck-Alignment equation
\begin{equation*}\label{}
\begin{split}
\partial_t f + v \cdot \n_x f &+ \a \nabla_v \cdot ( \st_\rho (\vavg - v) f ) = \s \D_v f,\\
\ddt \bbv &=  \int_\O \st_\rho (\vavg - v) \rho \dx = 0.
\end{split}
\end{equation*}
For our system \eqref{e:FPNSe}, on the other hand, the fluid and particle total momenta are influenced by the drag force,
\begin{equation}\label{ }
\ddt \bbu = \frac{\g}{|\O|} \int_\O \rho( \bv - \bu_\e) \dx, \quad \ddt \bbv = - \b \int_\O \rho( \bv - \bu_\e) \dx.
\end{equation}
A straightforward computation reveals  two conservations laws
\begin{align}
\dot{X}_1=0, \quad X_1 & =   \g \bbv + \b |\O| \bbu, \label{e:X1} \\
\dot{X}_2=0, \quad X_2 &=  \frac12 | \bbv|^2 +  \frac{\b|\O|}{2\g} |\bbu|^2 - \frac{\b |\O|}{2(\g +  \b |\O|)} |\bbu - \bbv|^2. \label{e:X2}
\end{align}

From the basic stochasticity relation \eqref{e:ker-s} it follows that all symmetric models, $\phi_\rho (x,y) = \phi_\rho (y,x)$, are conservative. 
As shown in \cite{Shv-EA} all conservative models are  {\em $p$-contractive} for all $1 \leq p \leq \infty$, meaning that for any  $\rho \in \cP(\O)$, $\bv\in L^p(\k_\rho)$,
\begin{equation}\label{e:contractive}
\| \ave{\bv}_\rho \|_{L^p(\k_\rho)} \leq \| \bv \|_{L^p(\k_\rho)}.
\end{equation}
In fact, contractivity for all $p$ is equivalent to contractivity for $p=1$ and is equivalent to being conservative, see \cite[Lemma 3.9]{Shv-EA}.

It is worth noting that all conservative models automatically satisfy Schur's condition \eqref{e:ker-Schur} due to the uniform bound on the strength function,
\begin{equation}\label{}
\sup_{\rho\in \cP} \left\| \int_\O \phi_\rho(x,\cdot) \drho(x) \right\|_\infty \leq \oS, \quad \sup_{\rho\in \cP} \left\| \int_\O \phi_\rho(\cdot,y) \drho(y) \right\|_\infty  \leq \oS.
\end{equation}
As far as the weak continuity condition \eqref{e:ker-cont}, it clearly holds if $\tilde{\rho}_n = \rho_n$, $\tilde{\rho} = \rho$ due to the continuity of the strength function \eqref{e:s-cont}. However, in general it has to be verified separately.

As a consequence of contractivity and the energy inequality \eqref{e:enineq}, we conclude that the energy for conservative systems is a Lyapunov function. 

\begin{lemma}\label{l:Lyap1} If the model $\cM$ is conservative, then the energy $\cE$ of any weak solution constructed in \thm{t:weak} is not increasing.
\end{lemma}

As a consequence of the uniform control on the energy several other quantities become controlled uniformly over time, which will have important implications to the long time dynamics and hypocoercivity.  
First, there is no aggregation of the density.

\begin{lemma}\label{l:Lyap2} Suppose the model $\cM$ is conservative, then for any  weak solution constructed in \thm{t:weak}, $\int_\Omega \rho \log \rho \dx$ remains uniformly bounded in time, and there exist $c_1,c_2 >0$  depending only on the initial condition and parameters of the system    such that 
\[
|  \{ \rho \geq c_1 \}|  \geq c_2.
\]
\end{lemma}
\begin{proof}
Since the relative entropy about the local Maxwellian remains non-negative
\begin{equation*}
	0 \leq \int_\domain f \log \frac{f}{\mu_{\rho}} \dv \dx,
\end{equation*}
expanding the right-hand side, we obtain
\begin{equation}
\label{eqn:rhologrho}
	\s \int_\Omega \rho \log \rho \dx \leq \s \int_\domain f \log f \dx \dv + \frac12 \int_\domain |v|^2 f \dv \dx + c(\s,n) = \cH +  c(\s,n). 
\end{equation}
Thus $\rho \in L\log L$ uniformly.

 Denote $A = \{ \rho \geq c_1 \}$.  
Using Jensen's inequality applied to the function $\psi(x) = x \log x$,
\[
	\psi \Big( \frac{1}{|A|} \int_{A} \rho \dx \Big) \leq \frac{1}{|A|} \int_{A} \psi(\rho) \dx 
\]
we get
\begin{equation}\label{e:AJens}
	\frac{1}{|A|} \Big( \int_{A} \rho \dx  \Big) \log \Big( \frac{1}{|A|} \int_{A} \rho \dx  \Big)  \leq \frac{1}{|A|} \int_{A} \rho \log \rho \dx \leq \frac{C}{|A|}.
\end{equation}

Since $\int_A \rho \dx = 1 - \int_{A^c} \rho \dx \geq 1 - c_1 |\Omega|$, we can choose $c_1$ small enough so that $\int_A \rho \dx \geq 1/2$. If in this case $\frac{1}{|A|} \int_{A} \rho \dx \leq 1$, then $|A| \geq 1/2$ and we are done. Otherwise, the logarithm above is positive,  so we obtain
\[
\frac{1}{2|A|} \log \Big( \frac{1}{|A|} \int_{A} \rho \dx  \Big) \leq \frac{C}{|A|},
\]
and hence
\[
 \log \Big( \frac{1}{|A|} \int_{A} \rho \dx  \Big) \leq 2C,
 \]
 which implies 
\[
	 \frac{1}{2 |A|} \leq \frac{1}{|A|} \int_{A} \rho \dx  \leq e^{2C},
\]
and, in particular, 
\begin{equation}
	\label{ineq:no_aggregation}
	\frac{1}{2} e^{-2C} \leq |A|.
\end{equation}

\end{proof}

Second, if a model is type-$(2,\infty)$, then it provides a uniform bound in time not only for the entropy $H$ but also for the drift $B$ as defined in \eqref{e:BH}. Consequently, applying \prop{p:Gauss} repeatedly on time intervals $[n,n+2)$ the solution gains positivity \eqref{e:Gauss} uniformly in time.  This implies a uniform bound from below on the density and, hence, on the thickness.
\begin{proposition}\label{p:thick}
Let $\cM$ be a regular conservative type-$(2,\infty)$ model. Then any weak solution constructed in \thm{t:weak} becomes uniformly thick
\[
\th(\rho, \O) \geq c(a,b,t_0).
\]
from any $t\geq t_0>0$, and regularizes instantly into any $H^m_q$ as in \eqref{e:gh}.
\end{proposition}

Another property that plays a crucial role in hypocoercivity analysis, specifically in the computations of spectral gaps, is the so-called ball-positivity:
\begin{equation}\label{e:bpenergies}
(\bv, \ave{\bv}_\rho)_{\k_\rho} \geq \| \ave{\bv}_\rho \|_{L^2(\k_\rho)}^2, \qquad \forall \bv\in L^2(\k_\rho).
\end{equation}
Ball-positive models on our list are \ref{CS}, provided the kernel is Bochner-positive $\phi = \psi \ast \psi$, $\psi \geq 0$, and \ref{Mf}, \ref{Mseg}. It is somewhat non-trivial to show that ball-positive models are conservative, see \cite[Proposition 3.13]{Shv-EA}. We thus have the following implication table.

\medskip
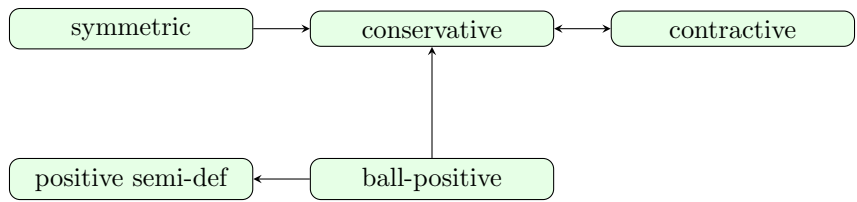
\begin{center}
\begin{tikzpicture}

\tikzstyle{boxr} = [rectangle, rounded corners,text centered, text width=3cm, draw=black, fill=green!10]

\tikzstyle{boxg} = [rectangle, rounded corners,  text centered, text width=3cm, draw=black, fill=green!10]

\tikzstyle{boxgwide} = [rectangle, rounded corners,  text centered, text width=4cm, draw=black, fill=green!10]

\tikzstyle{empty} = [ellipse,  text centered, draw=black]

\tikzstyle{arro} = [-,>=stealth]

\tikzstyle{arrow} = [->,>=stealth]
\tikzstyle{arrow2} = [<->,>=stealth]

\node (cons) at (4,0)  [boxg] {conservative};
 
\node (contr) at (8,0)  [boxg] {contractive};

\draw[arrow2] (cons) -- (contr);

\node (sym) at (0,0)  [boxg] {symmetric};
\draw[arrow] (sym) -- (cons);

\node (bp) at (4,-2)  [boxr] {ball-positive};
\node (psd) at (0,-2)  [boxg] {positive semi-def};

\draw[arrow] (bp) -- (psd);
\draw[arrow] (bp) -- (cons);

\end{tikzpicture}

\end{center}
\medskip

It is clear that the equality in \eqref{e:bpenergies} is always achieved on constant fields $\bv =\bbv$.  If $\bv$ is orthogonal to constants, i.e. $\bbv = 0$, we may expect a slightly more coercive property 
\begin{equation}\label{e:bpmore}
(\bv, \ave{\bv}_\rho)_{\k_\rho} \geq (1+\tilde{\d}) \| \ave{\bv}_\rho \|_{L^2(\k_\rho)}^2, \qquad \forall \bv\in L^2(\k_\rho), \bbv = 0.
\end{equation}
This translates into a crucial spectral gap inequality
\begin{equation}\label{e:spgap}
	\| \bv \|_{L^2(\kappa_\rho)}^2 - (\bv, \vavg)_{\kappa_\rho} \geq \d \|\bv\|_{L^2(\kappa_\rho)}^2, 
\end{equation}
for some $\d>0$. The utility of the spectral gap condition will be explained in the next section. It is important to note that for all ball-positive models on our list, the size of the spectral gap $\d>0$ depends only on the uniform thickness of the flock, see \cite[Proposition 4.16]{Shv-EA}. The table below gives a summary of the estimates from below on the spectral gap up to a constant depending only on the parameters of the system.
\medskip
\begin{center}
\begin{tabular}{  c |  c | c |c } 
 { Model} &  \ref{CS} &  \ref{Mf} &  \ref{Mseg}  \\ 
 \hline
\multirow{2}{*}{spectral gap} &\multirow{2}{*}{ $ \th^3 $ }&  \multirow{2}{*}{$ \th$ }& \multirow{2}{*}{$ \th^{2L}$}\\
 & & & 
 \end{tabular}
\end{center}

\subsection{Centered entropy} 
\label{ss:centered-entropy}
For conservative models the difference of energies present in the basic entropy law \eqref{e:enlaw} can be expressed in terms of variances from the total momentum $\bbv$: 
\[
\| \bv  \|_{L^2(\kappa_\rho)}^2 - (\bv , [\bv]_\rho)_{\kappa_\rho} = \| \bv -\bbv  \|_{L^2(\kappa_\rho)}^2 - (\bv-\bbv, [\bv -\bbv]_\rho)_{\kappa_\rho} .
\]
Moreover, since the fluid enstrophy (by the Poincare inequality) controls the correspondence variance of fluid velocity $|\bu - \bbu|^2$,
 we need to look into the entropy-energy function centered on the corresponding momenta:
\[
\ocE =  \sigma  \int_{\domain} f \log \frac{f}{\mu_{\bbv}} \dv \dx  + \frac{\b}{2\g} \int_\Omega |\bu - \bbu|^2 \dx + \frac{\b |\O|}{2(\g +  \b |\O|)} |\bbu - \bbv|^2.
\]
Note that expanding on the kinetic energies, the difference between $\cE$ and the previously considered entropy is given by $X_2$,
\begin{equation}\label{ }
\cE = \ocE + X_2,
\end{equation}
which is a conserved quantity. As a result, for any weak solution constructed in \thm{t:weak} the  centered entropy satisfies the same inequality,
\begin{equation}\label{law:shifted_energy}
\ddt \ocE \leq  - \a I_{vv}^{\bv, \st_\rho}  - \b I_{vv}^{\bu_\e} 
	 -  \a \| \bv -\bbv  \|_{L^2(\kappa_\rho)}^2 +\a (\bv-\bbv , [\bv -\bbv]_\rho)_{\kappa_\rho}  -   \frac{\b \nu}{\g} \int_\Omega |\nabla \bu|^2 \dx.
\end{equation}

Suppose  for a moment that our flock is uniformly thick in time and the spectral gap $\d$ as in \eqref{e:spgap} depends only on the thickness of the flock. Then keeping only the spectral gap term in the energy law \eqref{e:enlaw}, we obtain
\begin{equation}
	\label{law:H1}
	\ddt \ocE
		 \leq - \d  \| \bv - \bbv \|_{L^2(\kappa_\rho)}^2.
\end{equation}
On the other hand, using the identity
\[
 \| \bv - \bbv \|_{L^2(\kappa_\rho)}^2 + I_{vv}^{\bv, \st_\rho} = I_{vv}^{\bbv, \st_\rho},
\]
we rewrite the entropy equation as follows
\begin{equation}
	\label{law:H2}
	\ddt \ocE
		\leq   - \a I_{vv}^{\bbv, \st_\rho} - \b I_{vv}^{\bu_\e} - \nu \int_\Omega |\nabla \bu|^2 \dx +  \a (\bv - \bbv, [\bv - \bbv]_\rho)_{\kappa_\rho} .
\end{equation}
Since $(\bv - \bbv, [\bv - \bbv]_\rho)_{\kappa_\rho} \leq   \| \bv - \bbv \|_{L^2(\kappa_\rho)}^2$,  and $c\leq \st_\rho \leq \oS$, 
combining the two laws \eqref{law:H1} and  \eqref{law:H2}, we obtain 
\begin{equation}
	\label{law:H3}
	\begin{split}
	\ddt \cE
		&\lesssim  -I_{vv}^{\bbv} - I_{vv}^{\bu_\e} -  \int_\Omega |\nabla \bu|^2 \dx -\int_\Omega  |\bv - \bbv|^2 \rho \dx.
	\end{split}
\end{equation}
Let us now link the right hand side of the obtained entropy law to the macroscopic components of $\ocE$. First, by the Poincare inequality, we obtain
\[
 \int_\Omega |\bu - \bbu|^2 \dx  \lesssim \int_\Omega |\nabla \bu|^2 \dx .
\]

We now use the Fisher informations via to bound macroscopic energies. 
To this end observe that for any field $\bg= \bg(x)$ we have
\begin{equation}\label{ }
I_{vv}^{\bg} = I_{vv}^{\bv} + \int_\O | \bv - \bg|^2 \rho \dx  \geq \int_\O | \bv - \bg|^2 \rho \dx.
\end{equation}

Using the above with $\bg = \bu_\e$  
we obtain, as a consequence of \eqref{law:H3},
\[
	\ddt \cE
		\lesssim -I_{vv}^{\bbv} - \int_\Omega  |\bv - \bbv|^2 \rho \dx - \int_\Omega |\bv - \bu_\e|^2 \rho \dx- \int_\Omega |\bu - \bbu|^2 \dx .
\]


The difference of momenta $|\bbu - \bbv|^2$ can be recovered from the three macroscopic energies using the control on the concentration of density stated in \lem{l:Lyap2}.
We also use the fact that $\int_{\Omega} |\bu_\e - \bbu|^2 \dx = \int_{\Omega} |(\bu - \bbu)_\e|^2 \dx \leq \int_{\Omega} |\bu - \bbu|^2 \dx$.
By the  Chebyshev inequality,
\begin{align*}
	c_1 c_2 |\bbu - \bbv|^2 & \leq c_1 |\bbu - \bbv|^2 | \{ \rho \geq c_1 \}|
		=  c_1 \int_{ \{ \rho \geq c_1 \} } |\bbu - \bbv|^2  \dx \\
		&\leq  \int_{ \{ \rho \geq c_1 \} } |\bbv - \bu_\e|^2 \rho \dx +  c_1 \int_{ \{ \rho \geq c_1 \} } |\bu_\e - \bbu|^2  \dx \\
		&\lesssim \int_\Omega |\bbv - \bu_\e|^2 \rho \dx +  \int_{\Omega} |\bu_\e - \bbu|^2 \dx \\
		&\lesssim \int_\Omega  |\bbv - \bv|^2 \rho \dx +\int_\Omega |\bv - \bu_\e|^2 \rho \dx+  \int_{\Omega} |\bu - \bbu|^2 \dx  . 
\end{align*}

Finally,  the entropy law takes a form suitable for subsequent hypocoercivity analysis
\begin{equation}\label{e:eepre}
	\ddt \ocE  
		\lesssim  -I_{vv}^{\bbv}  -\int_\Omega  |\bv - \bbv|^2 \rho \dx-\int_\Omega |\bu - \bbu|^2 \dx-  |\bbu - \bbv|^2.
\end{equation}

\subsection{Hypocoercivity}
\label{ss:hypo}

\begin{theorem}\label{thm:main}
Suppose  $\cM$ is a regular conservative type-$(2,\infty)$ model of Schur's class  and the size of the spectral gap \eqref{e:spgap} depends only on the thickness $\d = \d(\th)$. Then any weak solution $(f,\bu)$ constructed in \thm{t:weak} with sufficiently high moment $q >n+4$, relaxes exponentially fast to a Maxwellian distribution for $f$ and uniform velocity for $\bu$ synchronized  to a common vector $\bbw \in \R^n$:
\[
\| f- \mu_{\bbw} \|_1 + \| \bv - \bbw\|_{L^2(\rho)}+ \| \bu - \bbw \|_2 \leq c_1e^{-c_2 t}.
\]
The constants $c_1,c_2$ depend only on the initial condition and the parameters of the system. The vector $\bbw$ can be determined from the initial condition 
\[
\bbw = \frac{\g \bbv_0 + \b |\O| \bbu_0}{\g + \b |\O|}.
\]
In particular, the statement applies to the following models:
\begin{itemize}
\item \ref{CS} with Bochner-positive kernel $\phi = \psi \ast \psi$;
\item \ref{Mf} with $\phi>0$;
\item \ref{Mseg} provided $\supp g_l = \O$.
\end{itemize}
\end{theorem}
\begin{proof}
Due to the assumptions on the model $\cM$ the statements of \prop{p:gh} and \prop{p:thick} apply to any weak solution constructed in \thm{t:weak}.  We thus have sufficient regularity to justify the analysis below and due to the spectral gap assumption the centered energy satisfies \eqref{e:eepre}.

To proceed further for convenience let us shift the kinetic velocity variable by the total momentum
\[
\tilde{f}(x,v,t) = f(x,v+\bbv,t), \quad \tilde{\bv} = \bv - \bbv, \quad \tilde{\rho} = \rho.
\]
Dropping tildes, the kinetic equation in the new variables reads
\[
 \partial_t f + (v + \bbv) \cdot \nabla_x f +  \a \nabla_v \cdot ( \st_\rho (\vavg - v) f )  = \bbv_t \cdot \n_v f+  \b \nabla_v \cdot \big( (v + \bbv - \bu_\e) f \big) 
  +  \sigma ( \b + \a \st_\rho) \Delta_v f .
\]
Note that in the new variables the total momentum vanishes, and the Fisher information becomes centered at $0$.  We rewrite the energy inequality  \eqref{e:eepre} as follows
\begin{equation}\label{e:eehypo}
	\ddt \ocE  
		\lesssim  -I_{vv}^{0}  -\| \bv\|^2_{L^2(\rho)} -\|\bu - \bbu\|^2_2  -  |\bbu - \bbv|^2.
\end{equation}

From now on,  for notational simplicity, we will assume that  $\s = 1$. 
Let us write the equation for the new distribution $h = \frac{f}{\mu}$. For this purpose it is simpler to view the equation as a Fokker-Planck equation with a drift similar to \eqref{e:FPbs},
\begin{equation}\label{ }
\begin{split}
  \partial_t f + (v + \bbv) \cdot \nabla_x f  &= \bbv_t \cdot \n_v f + \bb \cdot \nabla_v  f  +\st \nabla_v \cdot (  \n_v f + vf ) ,\\
\bb =  \b \bbv -\b \bu_\e & - \a \st_\rho \vavg,  \quad  \st = \b + \a\st_\rho.
  \end{split}
\end{equation}
Then the $h$-equation reads
\begin{equation}\label{e:FPAh}
h_t + (v + \bbv) \cdot \n_x h = \bbv_t \cdot \n_v h - v \cdot \bbv_t h+ \st (  \D_v h - v \cdot \n_v h) + \bb \cdot \n_v h - v \cdot \bb h.
\end{equation}
 Denoting
\[
B = (v + \bbv) \cdot \n_x, \quad A =  \n_v, \quad A^* = (v  -  \n_v) \cdot,
\]
where $A^*$ is understood relative to the inner product of the weighted space $L^2(\mu)$,
we arrive at 
\begin{equation}\label{e:FPALN}
h_t =  - \st A^* A h- Bh - A^*( \bb h) - A^*( \bbv_t h).
\end{equation}

Let us introduce the full set of Fisher informations
\[
\cI_{vv} =  \int_\domain \frac{|\n_v h|^2}{h} \dmu,
\quad  \cI_{xv} =  \int_\domain \frac{\n_x h \cdot \n_v h}{h} \dmu, \quad  \cI_{xx}=  \int_\domain \frac{|\n_x h|^2}{h} \dmu,
 \]
 The  information that captures full gradient  
 \begin{equation}\label{e:fF}
\cI = \cI_{vv} +\cI_{xx} 
\end{equation}
dominates the relative entropy by the classical log-Sobolev inequality
\[
\cI_{vv} +\cI_{xx} \geq \l \cH.
\]

We will now apply the estimates obtained in \cite{Shv-EA} with the inclusion of an abstract drift term $\bw$ which results in the following list:
\begin{equation*}\label{}
\begin{split}
\ddt \cI_{vv} & \leq  -  \cD_{vv}-  c_0 \cI_{vv} -2  \cI_{xv} +  2  (\bb, \bv)_\rho\\
\ddt \cI_{xv}  &\leq - \frac18  \cI_{xx} + c_1 \cI_{vv}+ \frac12   \cD_{vv} + \frac12  \cD_{xv} + c_2 \| \n_x \bb \|^2_{L^2(\rho)}  + c_3 \| \bb\|^2_{L^2(\rho)} \\
\ddt \cI_{xx}  &\leq   c_4 \cI_{vv} - \cD_{xv} + c_5 \| \n_x \bb \|^2_{L^2(\rho)},
\end{split}
\end{equation*}
where 
\[
\cD_{\bullet} = \int_\domain \st h |\n_\bullet \bar{h}|^2 \dmu, \quad \bar{h} = \log h.
\]

With the assumptions on the alignment drift we have 
\[
\| \n_x (\st_\rho \ave{\bv}_\rho) \|^2_{L^2(\rho)} \leq C \|  \st_\rho \ave{\bv}_\rho \|^2_{L^2(\rho)} \leq C \| \bv \|^2_{L^2(\rho)}
\]
and
\[
\| \n_x \bu_\e \|^2_{L^2(\rho)} \leq c_\e \|\bu - \bbu\|_2^2.
\]
Furthermore, in the first equation we split $\cI_{xv} \leq \frac{1}{16} \cI_{xx} + 16 \cI_{vv}$. 
Incorporating these bounds into the system above we obtain
\begin{equation*}\label{}
\begin{split}
\ddt \cI_{vv} & \leq  -  \cD_{vv} + c_0 \cI_{vv} + \frac{1}{16} \cI_{xx} +  C_1 \| \bv \|^2_{L^2(\rho)} + C_2   \|\bu - \bbu\|_2^2\\
\ddt \cI_{xv}  &\leq - \frac18  \cI_{xx} + c_1 \cI_{vv}+ \frac12   \cD_{vv} + \frac12  \cD_{xv}   + C_3 \| \bv \|^2_{L^2(\rho)} + C_4   \|\bu - \bbu\|_2^2 \\
\ddt \cI_{xx}  &\leq   c_4 \cI_{vv} - \cD_{xv} + C_5 \| \bv \|^2_{L^2(\rho)} + C_6   \|\bu - \bbu\|_2^2 ,
\end{split}
\end{equation*}
Adding up the above we absorb all the dissipation terms $\cD$'s and obtain

\[
\ddt (\cI_{vv}+\cI_{xv} +\cI_{xx} ) \leq c_1 \cI_{vv}  - c_2 \cI_{xx}  + c_3 \| \bv \|^2_{L^2(\rho)} + c_4  \|\bu - \bbu\|_2^2
\]

Adding this to a large constant multiple of \eqref{e:eehypo} we obtain
\[
\ddt ( \cI_{vv}+\cI_{xv} +\cI_{xx} + C \ocE) \leq - c_1  (\cI_{vv}  + \cI_{xx}) - c_2 \| \bv \|^2_{L^2(\rho)} - c_3  \|\bu - \bbu\|_2^2 - c_4  |\bbu - \bbv|^2.
\]
Denoting $Y = \cI_{vv}+\cI_{xv} +\cI_{xx} + C \ocE$ we observe that
\[
Y \sim \cI_{vv}+\cI_{xx} + a \| \bv \|^2_{L^2(\rho)} + b \|\bu - \bbu\|_2^2  + c  |\bbu - \bbv|^2.
\]
So, from the above we obtain 
\[
\ddt Y \lesssim - Y.
\]
This proves exponential relaxation.

\end{proof}


%
%
%
%
%
%
%
%
%
%
%
%

\end{document}